\documentclass[letterpaper]{amsart}
\usepackage{amsmath,amsthm,amsfonts}
\usepackage{amssymb}
\usepackage[T1]{fontenc}
\usepackage[utf8]{inputenc}
\usepackage[margin=1in]{geometry}
\usepackage{subcaption}
\usepackage{url}

\newcommand{\ev}[1]{(#1)}

\usepackage{esint}
\usepackage{graphicx}
\usepackage{units}
\usepackage[bookmarks=true,hidelinks]{hyperref}
\usepackage{cleveref}
\usepackage{bbm}
\usepackage{bm}
\usepackage{xcolor}
\usepackage{calrsfs}

\newtheorem{theorem}{Theorem}[section]
\newtheorem{proposition}[theorem]{Proposition}

\newtheorem{lemma}[theorem]{Lemma}
\newtheorem*{lemma*}{Lemma}

\theoremstyle{definition}
\newtheorem{definition}[theorem]{Definition}
\newtheorem{remark}[theorem]{Remark}

\numberwithin{equation}{section}

\def\XXint#1#2#3{{\setbox0=\hbox{$#1{#2#3}{\int}$ }
		\vcenter{\hbox{$#2#3$ }}\kern-.56\wd0}}

\newcommand{\dom}{D}
\renewcommand{\geq}{\geqslant}
\renewcommand{\leq}{\leqslant}

\renewcommand{\epsilon}{\varepsilon}

\newcommand{\R}{\mathbb{R}}
\newcommand{\N}{\mathbb{N}}

\newcommand{\Grad}{\nabla}
\newcommand{\Div}{\operatorname{div}}
\newcommand{\Curl}{\operatorname{curl}}
\newcommand{\curl}{\operatorname{curl}}

\newcommand{\weak}{\rightharpoonup}

\renewcommand{\geq}{\geqslant}
\renewcommand{\leq}{\leqslant}

\renewcommand{\div}{{\textrm{div}}}

\usepackage{todonotes}
\usepackage{enumerate}
\newcommand{\norm}[1]{\left\| #1 \right\|}

\newcommand{\HH}{L^2_{\Div}(\dom)}

\newcommand{\VV}{H^1_{\Div}(\dom)}
\newcommand{\E}{\mathbb{E}}
\newcommand{\T}{\mathbb{T}}
\renewcommand{\P}{\mathbb{P}}

\newcommand{\w}{\eta}
\DeclareMathOperator{\diam}{diam}

\title[Energy balance]{Sufficient Conditions for the Energy Balance for the Stochastic Incompressible Euler Equations with Additive Noise in two Space Dimensions}
\author[T. Rohner]{Tobias Rohner}
\address[Tobias Rohner]{\newline Seminar for Applied Mathematics (SAM) \newline ETH Z\"urich \newline R\"amistrasse 101 \newline Z\"urich, Switzerland.}
\email[]{tobias.rohner@sam.math.ethz.ch}
\author[F. Weber]{Franziska Weber}
\address[Franziska Weber]{\newline Department of Mathematics \newline University of California, Berkeley \newline Berkeley, CA 94720, USA.}
\email[]{fweber@math.berkeley.edu}
\date{\today}
\thanks{F.W. was partially supported by NSF DMS 2438083.}
\thanks{This work was supported by a computing grant from the Swiss National Supercomputing Centre (CSCS) under project ID 1217.}
\begin{document}
\maketitle
\begin{abstract}
    We consider vanishing viscosity approximations to solutions of the stochastic incompressible Euler equations in two space dimensions with additive noise. We identify sufficient and necessary conditions under which martingale solutions of the stochastic Euler equations satisfy an exact energy balance in mean. We find that the tightness of the laws of the approximating sequence of solutions of the stochastic Navier-Stokes equations in $L^2([0,T]\times \dom)$ is equivalent to the limiting martingale solution satisfying an energy balance in mean. Numerical simulations illustrate the theoretical findings.
	
\end{abstract}

\section{Introduction}
In this article, we consider the stochastic incompressible Euler equations on $\dom=\T^2$, and a time interval $t\in [0,T]$ for some $T>0$,
\begin{equation} \label{SEuler}
	\begin{split}
		du + u \cdot \nabla u \, dt& = -\nabla p \, dt + f \,dt + \sigma \cdot dW,\\
		\Div u & = 0,\\
		u(0,x) & = \bar{u}(x), 
	\end{split}
\end{equation} 
describing the motion of an inviscid, incompressible fluid on a periodic domain. Here $u:[0,T]\times\dom\to \R^2$ is the velocity field and $p:[0,T]\times\dom\to \R$ is the pressure. The
stochastic forcing is assumed to be a mean zero, white-in-time and colored-in-space Gaussian
process, that can be represented by
\begin{equation*}
\sigma \cdot W =\sigma(x) \cdot W(t, \omega) = \sum_{j=1}^\infty b_j e_j(x) \beta_j(t, \omega)	
\end{equation*}
 where $\beta_j$ are independent standard  1D Brownian motions defined on a complete probability space $(\Omega, \mathcal{F}, \P)$ with a filtration $(\mathcal{G}_t)_{t\geq 0}$. We assume that the $\sigma$-algebras $\mathcal{G}_t$ are completed with respect to $(\mathcal{F},\P)$, that is, $(\mathcal{G}_t)_{t\geq 0}$ contains all the $\P$-null sets $A\in \mathcal{F}$.  
 $\{e_j\}_{j=1}^\infty$ is a basis of Stokes eigenfunctions on $\HH$, where $\HH$ is the space of $L^2(\dom)$-valued functions that are divergence free almost everywhere.
 We assume furthermore the bounded trace conditions
\begin{equation}
	\label{eq:bounded trace}
	\sum_j b_j^2:=\bar{\sigma}<\infty,\quad \sum_j b_j^2\norm{\Curl e_j}_{L^2}^2:=\bar{\rho}<\infty.
\end{equation}
$f$ represents a deterministic forcing term satisfying $\Div f =0$ a.e. in $[0,T]\times\dom$. The stochastic term is added to account for uncertainties and random fluctuations that are inherent in real-world fluid flows. These could be environmental factors such as temperature variations, wind patterns, random external forces such as vibrations, etc. Furthermore, taking a probabilistic view is common in turbulence modeling and goes back to the famous work of Kolmogorov in 1941~\cite{K41a,K41b,K41c}.
The Euler equations describe the motion of an idealized fluid in the absence of friction and other diffusive forcing and are the formal vanishing viscosity limit of the stochastic Navier-Stokes equations. Formally, with a simple application of It\^o's formula, one can derive the energy balance in mean
\begin{equation}\label{eq:energybalanceintro}
	\E \int_{\dom} |u(t)|^2 dx =  \E\int_{\dom} |\bar{u}|^2 dx  + 2\E\int_0^t\int_{\dom} f(s)\cdot u(s) dx ds + \bar{\sigma } t.
\end{equation}
If $u$ is sufficiently smooth in space, this identity can be rigorously justified. However, when martingale solutions, which are the stochastic equivalent of weak solutions in the deterministic setting, are considered,~\eqref{eq:energybalanceintro} may only hold as an inequality. Moreover, energy dissipation is an integral part of Kolmogorov's K41 theory for 3D incompressible turbulence. Lars Onsager conjectured in 1949~\cite{Onsager1949} that deterministic three-dimensional turbulent flows conserve the energy law if the H\"older exponent $\theta$ of the solution is greater than 1/3, whereas anomalous dissipation of energy occurs when the H\"older exponent is less than 1/3. For the 3D deterministic Euler equations, the $\theta>1/3$ case has been settled a long time ago in the works by Eyink~\cite{Eyink1994} and Constantin, E and Titi~\cite{Constantin1994}. The $\theta<1/3$ part of the conjecture was finally proved in the recent work by Isett~\cite{Isett2018} building upon earlier work by Sheffer~\cite{Scheffer1993}, Shnirelman~\cite{Shnirelman1997}, and De Lellis and Sz\'ekelyhidi~\cite{DeLellis2009}, and others. In two space dimensions, the results of~\cite{Constantin1994} are readily applied to obtain energy conservation for $\theta>1/3$ in the force-free setting. One might hope that in two space dimensions, energy conservation might be possible under weaker conditions; however, this is not the case: Recenlty, Giri and Radu~\cite{Giri2024} constructed for any $\theta<1/3$ nontrivial weak solutions of the unforced Euler equations that satisfy $u\in C^\theta(\R_{\geq 0}\times \T^2)$ and do not conserve the energy.

In mathematical turbulence theory, solutions of the Euler equations are characterized as the zero viscosity limit of solutions of the Navier-Stokes equations. In fact, most numerical methods for approximating the Euler equations are built on the same paradigm: As the mesh size decreases, the numerical viscosity in the scheme vanishes. Therefore, it is interesting to determine conditions on the approximating sequence which guarantee that the limiting solution of the Euler equations satisfies an energy balance. In the 2D deterministic setting, such conditions were first investigated in the work~\cite{Cheskidov2016} in the unforced deterministic case. This work was extended to provide a complete characterization of energy-conservative solutions of the incompressible Euler equations that can be obtained as the zero viscosity of Navier-Stokes solutions in~\cite{Lanthaler2021}. In fact, it is shown that the zero viscosity limit that satisfies energy conservation is equivalent to the strong convergence of the sequence of solutions of the Navier-Stokes equations in $L^2([0,T]\times\dom)$ (where $[0,T]$ is an arbitrary time interval). The strong convergence in $L^2([0,T]\times\dom)$ again is equivalent to the Navier-Stokes solutions' time-averaged second-order structure functions having a uniform modulus of continuity.  Structure functions quantify how differences in a flow's velocity behave over a given distance or time separation, typically in terms of the difference in values at two points; for the precise definition used here, see Definition~\ref{def:structurefunction}. Later, these results were extended to the deterministically forced case in~\cite{Lopes2022,Jin2024}. In a recent work by De Rosa and Park~\cite{DeRosa2024}, it is furthermore shown that when the initial vorticity is a measure with positive singular part, any sequence of vanishing viscosity Leray-Hopf solutions to the periodic two-dimensional unforced incompressible Navier-Stokes equations does not display anomalous dissipation, thus further characterizing the dissipation mechanism in the case of weak convergence.
While these results hold in the two-dimensional periodic case, Ciampa~\cite{Ciampa2022} derived sufficient conditions for energy conservation in the case of the two dimensional plane.

As mentioned earlier, stochastic versions of the Navier-Stokes and Euler equations may be more natural to answer questions of turbulence than the corresponding deterministic models. However, the question of whether an energy balance in mean such as~\eqref{eq:energybalanceintro} holds for suitable notions of weak solutions of the stochastic Euler equations~\eqref{SEuler} that are obtained as vanishing viscosity limits from the Navier-Stokes equations does not appear to have been investigated yet. We therefore aim to start addressing this issue in this work and derive sufficient and necessary conditions under which an energy balance holds.
We find that if an averaged (in probability space) version of the structure function condition in~\cite{Lanthaler2021} holds, the laws of the Navier-Stokes solutions in the vanishing viscosity limit are tight in $L^2(0,T\times\dom)$. Thus, using Skorokhod's theorem, random fields can be identified such that the limit is a martingale solution of~\eqref{SEuler} and satisfies an energy balance in mean as in~\eqref{eq:energybalanceintro} with respect to the probability space coming from the application of Skorokhod's theorem. Conversely, we prove that if $u$ is a solution of the stochastic Euler equations that satisfies the energy balance~\eqref{eq:energybalanceintro} and is obtained as the vanishing viscosity limit of a sequence of solutions of the stochastic Navier-Stokes equations, then the laws of this approximating sequence are tight on $L^2([0,T]\times\dom)$. This is roughly a stochastic version of the equivalence of strong convergence and energy balance for the zero viscosity limit that was established in~\cite{Lanthaler2021}. We are able to adapt many of the technical tools of~\cite{Lanthaler2021,Jin2024} to the stochastic setting, however, the probabilistic setting added additional difficulty such as possibly unbounded trajectories. We resolved this by using a stopping time argument that allows us to deal with a.s. $L^2$-bounded realizations in which we can pass the viscosity to zero before finally sending the stopping time to infinity.

Although this work is restricted to additive smooth-in-space white noise, we hope that many of the techniques will also apply to other forms of noise, such as multiplicative or transport noise. We plan to study this in future work.

We conclude with numerical experiments with flat vortex sheet and fractional Brownian bridge initial data illustrating that the energy balance is close to being achieved in practice.

The rest of this article is organized as follows. In Section~\ref{sec:not}, we introduce the necessary notation and definitions of solutions. In Section~\ref{sec:balance}, we establish conditions for the energy balance in mean for solutions of~\eqref{SEuler} obtained as limits of solutions of the stochastic Navier-Stokes equations. Finally, in the last Section~\ref{sec:num}, we present numerical examples that confirm that the energy balance is maintained for rather rough initial conditions.

\section{Notation and preliminary results}\label{sec:not}
We start by introducing the notation and definitions that are needed later. We use the standard notation for Lebesgue and Sobolev spaces. We denote the subspaces of $L^2(\dom)$ and $H^1(\dom)$ consisting of almost everywhere divergence free functions as
$$\HH = \left\{ u \in L^2(\dom): \div u = 0 \right\},$$
$$\VV = \left\{ u \in H^1(\dom): \div u = 0 \right\}.$$
Sometimes, we will abbreviate $L^2 = L^2(\dom)$ and $H^1 = H^1(\dom)$. 
We will denote the $L^2(\dom)$ and $\HH$-inner product for functions or vector fields $f,g\in L^2(\dom)$
\begin{equation*}
	(f,g) = \int_{L^2(\dom)} f(x)\cdot g(x) dx.
\end{equation*}
For a Banach space $X$, we denote by $X_w$, the Banach space equipped with the topology of weak convergence. We say that a time-dependent function $f:[0,T]\to X$ is weakly continuous, denoted $f\in C(0,T;X_w)$, if the mapping $t\mapsto \langle f(t),\psi\rangle_{X,X^*}$ is continuous for any $\psi\in X^*$.
 Next, we define the space of smooth cylindrical functions $\text{Cyl}(X)$ for a Hilbert space $X$.
 \begin{definition}
 	We denote by $\Pi_d(X)$ the space of all maps $\pi:X\to \R^d$ of the form 
 	\begin{equation*}
 		\pi (x)  = (\langle x,e_1\rangle, \langle x,e_2\rangle, \dots, \langle x,e_d\rangle), \quad x\in X,
 	\end{equation*}
 	where $\{e_1,\dots,e_d\}$ is any orthonormal family of vectors in $X$. 
 	We denote by $\text{Cyl}(X)$ the functions $\varphi = \psi\circ \pi$ with $\pi\in \Pi_d(X)$ and $\psi\in C_c^\infty(\R^d)$.
 \end{definition}
 
 \begin{remark}\label{rem:weaktight}
 	Note that $\varphi=\psi\circ \pi\in \text{Cyl}(X)$ is a Lipschitz function and continuous with respect to the weak topology on $X$. By~\cite[Lemma 5.1.12]{Ambrosio2008}, we have that a sequence if 
 	$\mathcal{K}:=\{\mu_n\}\subset \mathcal{P}(X)$ is weakly tight in the sense  that
 	\begin{equation*}
 		\forall \epsilon >0\quad \exists R_\epsilon>0,\quad \text{such that }\quad \mu(X\backslash\overline{B_{R_\epsilon}})\leq \epsilon,\quad \forall \mu\in \mathcal{K}
 	\end{equation*}
 	it converges weakly to $\mu$ in $\mathcal{P}(X_{{w}})$ if and only if
 	\begin{equation}\label{eq:cylconv}
 		\lim_{n\to\infty}\int_X\varphi(x) d\mu_n(x) =\int_X \varphi(x) d\mu(x),\quad \forall\varphi\in \text{Cyl}(X).
 	\end{equation}
 \end{remark}
 Next, we define what we mean by solutions of the Euler equations~\eqref{SEuler} and of the Navier-Stokes equations which can be seen as the viscous approximation of the Euler equations. We start with the Navier-Stokes equations.
 \begin{definition}[Solutions of the stochastic Navier-Stokes equations~\cite{Kuksin2012}]\label{def:strongsolNS}
 	Let $T>0$. Let $f_\nu \in L^2([0,T); \HH)$ and $\bar{u}_\nu\in L^2(\Omega,\HH)$.  Let  $\{\beta_j\}_{j\in \N}$ be Brownian motions defined on a complete probability space $(\Omega,\mathcal{F},\P)$ with a filtration $\mathcal{G}_t$ $t\geq 0$ which is complete with respect to $(\Omega,\mathcal{F},\P)$,  and assume that $\bar{u}_\nu$ is measurable with respect to $\mathcal{G}_0$. Then an $\HH$-valued process $u_\nu(t)$, $t\geq 0$, is called a \emph{probabilistically strong solution} of the stochastic Navier-Stokes equations
 	 	\begin{equation} \label{SNS}
 	 		\begin{split}
 		du_\nu + u_\nu \cdot \nabla u_\nu \, dt& = -\nabla p_\nu \, dt + \nu \Delta u_\nu \,dt + f_\nu \,dt + \sigma \cdot dW \\
 		\Div u_\nu &= 0,\\
 		u_\nu(0,x) & = \bar{u}_\nu,
 	\end{split}
 	\end{equation}
 	for $t\in [0,T]$ if
 	\begin{enumerate}[(a)]
 		\item  The process $u_\nu(t)$ is adapted to the filtration $\mathcal{G}_t$, and almost every of its trajectories belongs to the space 
 		\begin{equation*}
 			C([0,T],\HH)\cap L^2([0,T],\VV).
 		\end{equation*}
 		\item For all $t\in [0,T]$, and all divergence free $\varphi\in C^\infty(\dom)$, the following identity holds $\P$-almost surely,
 			\begin{align}\label{def NS mart sol}
 			\begin{split}
 			&\int_{\dom} u_\nu(t)\cdot \varphi dx + \nu\int_0^t\int_{\dom} \Grad u_\nu(s): \Grad \varphi dx ds + \int_0^t\int_{\dom}\left[ (u_\nu(s)\cdot \Grad)u_\nu(s)\right]\cdot \varphi dx ds\\
 			&\qquad = \int_{\dom} \bar{u}_\nu \cdot \varphi dx +\int_0^t\int_{\dom} f_\nu \cdot \varphi dx ds+ \sum_k \int_0^t b_k e_k \cdot \varphi dx d\beta_k(s).
 			\end{split}
 		\end{align}
 		\item The initial condition $u_\nu(0)=\bar{u}_\nu$ is satisfied almost surely.
 	\end{enumerate}
 \end{definition}
 Existence and uniqueness of such solutions of the 2D-Navier-Stokes equations was shown in~\cite[Proposition 2.4.6]{Kuksin2012}. In~\cite[Propositon 2.4.8]{Kuksin2012}, it was also shown that such solutions satisfy the energy balance
 	\begin{equation}\label{eq:NSenergybalance}
 	\frac12\E\norm{u_\nu(t_2)}_{L^2}^2 -\frac12\E \norm{u_\nu(t_1)}_{L^2}+\nu\E\int_{t_1}^{t_2}\norm{\Grad u_\nu(s)}_{L^2}^2ds = \frac12\bar{\sigma} (t_2-t_1)+\E\int_{t_1}^{t_2} (f_\nu,u_\nu)ds.
 \end{equation}
 for every $t_2>t_1$, $t_1=0$ and almost every $t_1\in (0,\infty)$,
 as long as the initial data $\bar{u}_\nu$ has finite second moment.
 \begin{remark}\label{rem:gronwallbound}
 	The last term in the energy balance~\eqref{eq:NSenergybalance} can be bounded using Young's inequality:
 	\begin{equation*}
 		\left|\E\int_{t_1}^{t_2}(u_\nu,f_\nu)ds \right| \leq \frac12\E\int_{t_1}^{t_2} \norm{u_\nu}_{L^2}^2ds +  \frac12\E\int_{t_1}^{t_2} \norm{f_\nu}_{L^2}^2 ds.
 	\end{equation*}
 	Then using Gr\"onwall inequality, it follows that $\E\norm{u_\nu(t_2)}_{L^2}^2$ is bounded by $\E\norm{u_\nu(t_1)}_{L^2}^2$ times a constant depending on time, $f_\nu$ and $\bar{\sigma}$.
 \end{remark}
Next, we define solutions in the probabilistically strong sense and  martingale solutions for the Euler equations~\eqref{SEuler}. Solutions that are probabilistically strong are defined in the same way as for the Navier-Stokes equations:
\begin{definition}[Solutions of the stochastic Euler equations in the probabilistically strong sense~\cite{Kuksin2012}]\label{def:strongsolEuler}
	
	Assume that $f \in L^2([0,T); \HH)$ and $\bar{u}\in L^2(\Omega,\HH)$.  Let  $\{\beta_j\}_{j\in \N}$ be 1D Brownian motions defined on a complete probability space $(\Omega,\mathcal{F},\P)$ with a filtration $\mathcal{G}_t$ $t\geq 0$ which is complete with respect to $(\mathcal{F},\P)$,  and assume that $\bar{u}$ is measurable with respect to $\mathcal{G}_0$. Then an $\HH$-valued process $u(t)$, $t\geq 0$, is called a \emph{solution in the probabilistically strong sense} of the stochastic Euler equations~\eqref{SEuler},
	for $t\in [0,T]$ if
	\begin{enumerate}[(a)]
		\item  The process $u(t)$ is adapted to the filtration $\mathcal{G}_t$, almost every of its trajectories belongs to the space 
		\begin{equation*}
			C([0,T];H^\alpha(\dom))\cap L^\infty(0,T;\HH),\quad \text{for some }\, \alpha<0.
		\end{equation*}
		\item For all $t\in [0,T]$, and all divergence free $\varphi\in C^\infty(\dom)$, the following identity holds $\P$-almost surely,
		\begin{align*}
			\int_{\dom} u(t)\cdot \varphi dx - \int_0^t\int_{\dom}\left[ (u(s)\cdot \Grad)\varphi)\right]\cdot u(s) dx ds
		= \int_{\dom} \bar{u} \cdot \varphi dx +\int_0^t \int_{\dom} f\cdot \varphi dx ds + \sum_k \int_0^t b_k e_k \cdot \varphi dx d\beta_k(s).
		\end{align*}
		\item The initial condition $u(0)=\bar{u}$ is satisfied almost surely.
	\end{enumerate}
\end{definition}
Additionally, the solutions satisfy an energy inequality if $\bar{u}$ has a finite second moment:
	For every $t_2>t_1$, for $t_1=0$ and almost every $t_1\in (0,\infty)$,
	\begin{equation*}
		\frac12\E\norm{u(t_2)}_{L^2}^2 -\frac12\E \norm{u(t_1)}_{L^2} \leq \frac12\bar{\sigma} (t_2-t_1)+\E\int_{t_1}^{t_2}(f,u)  ds.
\end{equation*}
As in Remark~\ref{rem:gronwallbound}, this implies boundedness of $\E\norm{u(t)}_{L^2}^2$ uniformly in time if the initial data is in $L^2(\Omega;\HH)$.
In addition, there is a weaker solution concept when we are not able to deduce measurability of $u(t)$ with respect to the filtration $\{\mathcal{G}_t\}$ from the Brownian motion.  This is the type of solutions we will need later on. For these \emph{martingale solutions}, the probability space and the filtration are part of the solution concept.

\begin{definition}[Martingale solutions of Euler equations]\label{def:eulersol}
	  Let $f \in L^2([0,T); \HH)$.  A \emph{martingale solution} to the stochastic Euler equations~\eqref{SEuler}
	on $[0,T]$ consists of a stochastic basis $(\widetilde{\Omega},\widetilde{\mathcal{F}},(\widetilde{\mathcal{G}}_t)_{t\in[0,T]}, \widetilde{\P},\{\widetilde{\beta}_k\})$, where the $\{\widetilde{\beta}_k\}$ are Brownian motions with the same law as $\{\beta_k\}$, with a complete right-continuous filtration and an $\widetilde{\mathcal{G}}_t$-progressively measurable stochastic process
	\begin{equation*}
		u:[0,T]\times \Omega\to \HH,
	\end{equation*}
	such that
	\begin{itemize}
		\item $u$ has sample paths in 
		\begin{equation*}
			C([0,T];H^\alpha(\dom))\cap L^\infty(0,T;\HH),\quad \text{for some }\, \alpha<0.
		\end{equation*}
		\item For all $t\in [0,T]$, and all divergence free $\varphi\in C^\infty(\dom)$, the following identity holds $\widetilde{\P}$-almost surely,
		\begin{equation}	\label{eq:martingalesolutionseuler}
			\int_{\dom} u(t)\cdot \varphi dx - \int_0^t\int_{\dom}\left[ (u(s)\cdot \Grad)\varphi)\right]\cdot u(s) dx ds
		= \int_{\dom} \bar{u} \cdot \varphi dx +\int_0^t \int_{\dom} f\cdot \varphi dx ds + \sum_k \int_0^t b_k e_k \cdot \varphi dx d\widetilde{\beta}_k(s).
		\end{equation}
		\item $u$ satisfies the energy inequality in mean for every $t_2>t_1$, for $t_1=0$ and almost every $t_1\in (0,\infty)$,
		\begin{equation*}
			\frac12\widetilde{\E}\norm{u(t_2)}_{L^2}^2 -\frac12\widetilde{\E} \norm{u(t_1)}_{L^2} \leq \frac12\bar{\sigma} (t_2-t_1)+\widetilde{\E}\int_{t_1}^{t_2}(f,u)  ds,
		\end{equation*}
		here $\widetilde{\E}$ denotes the expectation with respect to $\widetilde{\P}$.
	\end{itemize}

\end{definition}

Note that for a martingale solution, the probability space is part of the solution definition. 	However, if there is no risk of confusion, we simply say that $u$ is a martingale solution without making reference to the stochastic basis.

Existence of martingale solutions to the 2D stochastic Euler equations was proved in~\cite{Bessaih2013,Bessaih1999b}.
Under the assumption that $f\in L^2(0,T;\VV)$ and $\bar {u}\in \VV$ a.e. $\omega\in \Omega$ and the trajectories of the Brownian motion $W$ satisfying $W(\cdot,\omega)\in C([0,T];H^4(\dom)\cap \VV)$, existence of a probabilistically strong solution with $u(\cdot,\omega)\in C(0,T;\HH)\cap L^2(0,T;\VV)$ for $\P$-almost every $\omega\in \Omega$ can be shown~\cite[Theorem 1.1]{Bessaih1999}. Note that while for  solutions in the probabilistically strong sense the probability space is given, they are not necessarily unique. However, if in addition $\Curl\bar{u}\in L^\infty(\dom)$ for almost every $\omega\in \Omega$, $\Curl f\in L^\infty([0,T]\times\dom)$ and $\Delta\Curl W\in L^\infty([0,T]\times\dom)$ for $\P$-almost every $\omega\in \Omega$, then the solution is unique~\cite[Theorem 1.2]{Bessaih1999}. This was improved slightly in~\cite{Kim2002} using a different approximation technique to obtain a unique pathwise solution under the condition that $W(\cdot,\cdot,\omega)\in C([0,T];H^{3+\alpha}(\dom))$ a.s. for some $\alpha>0$ and the existence of a pathwise solution under the condition that $u_0\in \VV$ is deterministic and $W$ is given as $W(t,x,\omega)=g(x) B(t,\omega)$ where $g\in \VV$ and $B$ is a standard one-dimensional  Brownian motion.

\begin{remark}
	Martingale solutions can be defined for the stochastic Navier-Stokes equations as well, but are not necessary for additive noise in 2D that we are considering here, see~\cite{Bensoussan1973} for the existence of solutions in the probabilistically strong sense in 2 and 3D and~\cite{Flandoli1995} for martingale solutions in 3D where the noise is multiplicative.
\end{remark}
We will also denote by $\w= \Curl u$ the vorticity of the solution $u$ of the Euler equations. Here the 2D curl is defined as $\Curl u = \partial_1 u^2-\partial_2 u^1$. Formally, in 2D, it solves the equation
\begin{equation}
	\label{eq:vorticityEuler}
	\begin{split}
		d\w + (u\cdot\Grad) \w dt &= \Curl f dt + \Curl\sigma \cdot dW,\\
		\Curl u & = \w,\\ 
		\w(0,x) & = \bar{\w}(x) = \Curl \bar{u}(x). 
	\end{split}
\end{equation}

Finally, we require the definition of the second-order structure function and the modulus of continuity.
\begin{definition}[Structure functions]\label{def:structurefunction}
	For $v \in L^2_x$ define the second-order \emph{structure function}
	$$S_2(v;r) = \left( \int_D \int_{B_r(0)} |v(x+h) - v(x)|^2 \, dh \, dx \right)^{1/2}.$$
	and the time-integrated structure function
	\begin{equation*}
		S_2^T(v;r) = \left(\int_0^T \int_D \int_{B_r(0)} |v(t,x+h) - v(t,x)|^2 \, dh \, dx \, dt \right)^{1/2}= \left(\int_0^T S_2(v(t),r)^2 dt\right)^{1/2}.
	\end{equation*}
\end{definition}
\begin{definition}[Modulus of continuity]
	\label{def:moc}
	
	A modulus of continuity is a function $\phi:[0,\infty)\to [0,\infty)$ to measure the continuity of functions quantitatively. $\phi$ is continuous and satisfies $\phi(0)=0$. We say a function $f:(X,d_X)\to (Y,d_Y)$, where $(X,d_X)$ and $(Y,d_Y)$ are two metric spaces, admits a modulus of continuity $\phi$ if
	\begin{equation*}
		d_Y(f(x),f(y))\leq \phi(d_X(x,y)),
	\end{equation*}
	for any $x,y\in X$.
\end{definition}
We will assume in the following that any modulus of continuity $\phi$ is monotone increasing, which will not be a restriction, since if $f$ admits any given modulus of continuity $\phi$, we can find a continuous  $\widetilde{\phi}$ that satisfies $\widetilde{\phi}(0)=0$ and $\widetilde{\phi}\geq \phi$ and is monotone increasing and hence a modulus of continuity for the same given function.
It turns out that if the structure functions $S_2(v_n,r)$ of a bounded sequence of functions $\{v_n\}_{n\in \N}\subset L^2(\dom)$ have a common modulus of continuity, the sequence is precompact in $L^2(\dom)$:
\begin{theorem}[Fr\'echet-Kolmogorov theorem]
	\label{prop:mocL2}
	Let $\{v_n\}_{n\in \N}$ be a bounded sequence of functions in $L^2(\dom)$ where $\dom\subset \R^d$ is bounded, $d\in \N$. Assume that for all $n\in \N$
	\begin{equation*}
		S_2(v_n,r)\leq \phi(r),\quad \forall \, r>0,
	\end{equation*}
	for some modulus of continuity $\phi$.  Then $\{v_n\}_{n\in \N}$ has a subsequence that is convergent in $L^2(\dom)$.
\end{theorem}
For a proof, see e.g.~\cite[Theorem A.8]{Holden2015}. If the structure functions $S_2^T(v_n,r)$ of a bounded sequence of time-dependent functions $\{v_n\}_{n\in \N}\subset L^2([0,T];L^2(\dom))$, where $\dom\subset \R^d$ is a bounded domain, have a uniform modulus of continuity and in addition some time regularity, then the sequence is precompact in $L^2([0,T];L^2(\dom))$:
\begin{proposition}\label{prop:moctimedep}
	Let $\{v_n\}_{n\in \N}$ be a bounded sequence of functions in $L^2(0,T;L^2(\dom))$, where $\dom\subset \R^d$ is bounded, $d\in\N$. Assume that for all $n\in \N$,
	\begin{equation*}
		S_2^T(v_n,r)\leq \phi(r),\quad \forall \, r>0,
	\end{equation*} 
	for some modulus of continuity $\phi$. Assume in addition that 
	\begin{equation*}
		\{v_n\}_{n\in \N}\subset W^{\beta_1,p_1}(0,T;H^{-L}(\dom))+ \dots + W^{\beta_m,p_m}(0,T;H^{-L}(\dom)),
	\end{equation*}
	uniformly for all $n\in \N$, and
	where 
	\begin{equation*}
		\beta_1,\dots, \beta_m\in (0,1],\quad p_1,\dots, p_m>1
	\end{equation*}
	satisfy $\beta_i p_i>1$ for all $i=1,\dots, m$. Then $\{v_n\}_{n\in \N}\subset C^\alpha(0,T;H^{-L}(\dom))$ uniformly in $n\in \N$, where $\alpha = \min_{i=1,\dots,m}(\beta_i-1/p_i)$ and  $\{v_n\}_{n\in \N}$ has a subsequence that is strongly convergent in $L^2(0,T;L^2(\dom))$.
\end{proposition}
\begin{proof}
	We first note that by the fractional version of Morrey's inequality, see e.g.~\cite{DiNezza2012}, we have that $W^{\beta_1,p_1}(0,T;H^{-L}(\dom))+ \dots + W^{\beta_m,p_m}(0,T;H^{-L}(\dom))$ embeds into $C^\alpha(0,T;H^{-L}(\dom))$ for $0<\alpha\leq \min_{i=1,\dots,m}(\beta_i-1/p_i)$ since each of the $W^{\beta_i,p_i}(0,T;H^{-L}(\dom))$ do. Then the proof follows as in, e.g.~\cite[Proposition 2.10]{Lanthaler2021}.
\end{proof}
\begin{remark}
	Equivalently, in Proposition~\ref{prop:moctimedep}, we could assume
	\begin{equation*}
		\sup_{0<r\leq r_0} \frac{S_2^T(v_n,r)}{\phi(r)}\leq 1,
	\end{equation*}
	for some $r_0>0$, 
	for all $v_n$, $n\in \N$, since for large $r$, $S_2^T(v_n,r)$ is trivially bounded by twice the $L^2$-bound on $v_n$.
\end{remark}
\section{Energy balance}\label{sec:balance}
\subsection{A sufficient condition for an energy balance in mean}
Our goal is now to show that if the sequence of Navier-Stokes solutions $\{u_\nu\}_\nu$ satisfy a uniform condition on the modulus of continuity of the structure functions in mean,
\begin{equation}
	\label{eq:structurefcn}
	\E \sup_{0< r\leq r_0}\frac{S_2^T(u_\nu,r)^2}{\phi(r)^2} \leq C ,
\end{equation}
for some $C>0$, $r_0>0$,
uniformly in $\nu>0$ for some modulus of continuity $\phi$, then the sequence of laws of $u_\nu$, $\mu_\nu = \P\circ u_\nu^{-1}$ is tight and with Skorokhod's theorem, a random variable and a probability space can be identified which satisfies an energy balance in mean with respect to that probability measure. Conversely, if $u_\nu$ is a physical realization obtained through a sequence of solutions of the stochastic Navier-Stokes equations, and its limit $u$ satisfies an energy balance in mean, then the sequence of laws $\mu_\nu$ is tight on $L^2(0,T;L^2_{\Div}(\dom))$.  Since the condition~\eqref{eq:structurefcn} combined with the uniform weak continuity in time of the $u_\nu$ yield tightness in $L^2(0,T;\HH)$, this is a stochastic version of the reverse direction.
\begin{theorem}
	\label{prop:tightnesstoenergybalance}
	Let $f\in L^2(0,T;\HH)$. 
	Let $\{u_\nu\}_{\nu>0}$ be a sequence of stochastically strong solutions of the 2D stochastic Navier-Stokes equations~\eqref{SNS} with initial data $\bar{u}_{\nu}\in L^2(\Omega;\HH)$ and forcing term $f_\nu\in L^2(0,T;\HH)$. 
	Assume that as $\nu\to 0$, $\bar{u}_{\nu}\to \bar{u}$ in $L^2(\Omega;\HH)$ and that $f_\nu\weak f$ in $L^2(0,T;\HH)$. 
	Assume that the initial data satisfies uniformly in $\nu>0$ for some $p>2$,
	\begin{equation*}
		\E\norm{\bar{u}_\nu}_{L^2}^p\leq C<\infty,
	\end{equation*}
	and assume that $\bar{u}_\nu$ are measurable with respect to $\mathcal{G}_0$.
	Furthermore, assume that the $u_\nu$ satisfy, uniformly in $\nu$, 
	\begin{equation}
		\label{eq:weakstructurefcn2}
		\E \sup_{0<r\leq r_0}\frac{S_2^T(u_\nu(t),r)^2}{\phi(r)^2}  \leq C ,
	\end{equation}
	for some  monotonically increasing modulus of continuity $\varphi:[0,\infty)\to [0,\infty)$ an some $r_0>0$. Then a subsequence of the laws $\mu_\nu=\P\circ u_\nu^{-1}$  of $\{u_{\nu}\}_{\nu>0}$ converges weakly on $L^2([0,T];\HH)$ and there exists a probability space $(\widetilde{\Omega},\widetilde{\mathcal{F}},\{\widetilde{\mathcal{G}}_t\}_{t\in [0,T]},\widetilde{\P})$ and random variables $\widetilde{u}_\nu$ that have the same laws as $u_{\nu}$ and converge  $\widetilde{\P}$-a.s. to $\widetilde{u}$ which is a martingale solution of the stochastic Euler equations as in Definition~\ref{def:eulersol} that satisfies for almost every $t\in [0,T]$, the energy balance
	\begin{equation*}
		\widetilde{\E}\norm{\widetilde{u}(t)}^2_{L^2} = \widetilde{\E}\norm{\widetilde{\bar{u}}}_{L^2}^2 + 2 \widetilde{\E} \int_0^t(\widetilde{u},f) ds + \bar{\sigma} t.
	\end{equation*}
\end{theorem}
To prove this, we need a couple of auxiliary lemmas whose deterministic versions were proved in~\cite{Lanthaler2021}. We show here how to modify them to hold in the stochastic case.
We start with the following enhanced Poincar\'e lemma which was proved in~\cite{Lanthaler2021}:
\begin{lemma}\label{lem:poincaremean}
	Let $v$ be a random variable on $(\Omega,\mathcal{F},(\mathcal{G}_t)_{t\in [0,T]},\P)$  taking values in $H^2(\dom)\cap \HH$ almost surely and denote $\w: = \Curl v$. Assume that 
	\begin{equation}\label{eq:ass1}
		\left(\E S_2^T(v;r)^2\right)^{1/2}\leq C \phi(r),\quad \forall \, 0\leq r\leq r_0
	\end{equation}
	for some monotonically increasing modulus of continuity and a suitable $r_0>0$. Furthermore, let $\tau$ be a stopping time that is bounded almost surely, i.e., $\tau\leq C$ a.s. for some $C>0$. Then for any $\delta>0$
	\begin{equation}
		\label{eq:poincare}
	\left(\E\int_{\delta\wedge \tau}^{T\wedge \tau}\!\!\int_{\delta\wedge \tau}^{t\wedge\tau} \norm{ \w_\nu}^2_{L^2} \, dsdt\right)^2 \gamma\left(\E\int_{\delta\wedge \tau}^{T\wedge \tau}\!\!\int_{\delta\wedge \tau}^{t\wedge\tau} \norm{ \w_\nu}^2_{L^2} \, dsdt\right) \leq\E\int_{\delta\wedge \tau}^{T\wedge \tau}\!\!\int_{\delta\wedge \tau}^{t\wedge\tau} \norm{\Grad \w_\nu}^2_{L^2} \, dsdt, 
	\end{equation}
	where $\gamma:\R_{\geq 0}\to\R_{\geq 0}$ is a continuous monotonically increasing function with the property that $\gamma(y)\geq \gamma_0>0$ for all $y\geq 0$ and $\gamma(y)\to \infty$ as $y\to \infty$.
\end{lemma}
\begin{proof}
	The key idea of the proof is Lemma 2.6 in~\cite{Lanthaler2021}, i.e., we have to show that for a function $v\in H^2(\dom)\cap \HH$ with $\Curl v = \w$ we have
	\begin{equation}
		\label{eq:ineqformega}
	\norm{\w}_{L^2}^2 \leq 2 S_2(v;r)^2+ C r^2  \norm{\Grad \w}_{L^2}^2,
		\end{equation}
	To show this inequality, we assume that $v$ is smooth, then the result follows by using the density of smooth functions in $H^2(\dom)\cap \VV$. By Taylor's theorem, we have for $h\in \R^2$
	\begin{equation*}
		h\cdot \Grad v(x) = v(x+h)-v(x) - \int_0^1 (1-t) (h\otimes h ): \Grad^2 v(x+th) dt.
	\end{equation*}
	By triangle inequality, we have
	\begin{equation*}
		|h\cdot \Grad v(x)|^2 \leq  2|v(x+h)-v(x)|^2 + 2\left|\int_0^1 (1-t) h\otimes h : \Grad^2 v(x+th) dt\right|^2.
	\end{equation*}
	Now we integrate over $B_r(0)\times \dom$ and divide by $|B_r(0)|$:
	\begin{align*}
		\int_{\dom}\fint_{B_{r}(0)}|h\cdot \Grad v(x)|^2dh dx &\leq 2\int_{\dom}\fint_{B_{r}(0)} |v(x+h)-v(x)|^2dh dx + 2\int_{\dom}\fint_{B_{r}(0)}\left|\int_0^1 (1-t) h\otimes h : \Grad^2 v(x+th) dt\right|^2dh dx\\
		& \leq 2 S_2(v;r)^2+ 2\int_{\dom}\fint_{B_{r}(0)} |h|^4 |\Grad^2 v(x)|^2dh dx\\
		& \leq 2 S_2(v;r)^2+ C r^4  \norm{\Grad^2 v}_{L^2}^2.
	\end{align*}
	We note that due to the periodic boundary conditions and the fact that $v$ is divergence free, we have $\norm{\Grad^2 v}_{L^2}= \norm{\Delta v}_{L^2} = \norm{\Curl \w}_{L^2}=\norm{\Grad\w}_{L^2}$. Furthermore, we can rewrite the integrand on the left hand side as follows using spherical coordinates:
	\begin{align*}
		\fint_{B_{r}(0)}|h\cdot \Grad v(x)|^2 dh & = \frac{1}{\pi r^2}\int_0^r\int_0^{2\pi}s\left[(s\sin\theta\partial_1 v^1+s\cos\theta\partial_2 v^1)^2 + (s\sin\theta\partial_1v^2 +s\cos\theta\partial_2 v^2)^2\right]d\theta ds\\
		&=\frac{1}{\pi r^2}\int_0^r s^3 \int_0^{2\pi}\Bigg[ \left(\frac12-\frac{\cos(2\theta)}{2}\right)\left((\partial_1 v^1)^2+(\partial_1 v^2)^2\right)+\left(\frac12+\frac{\cos(2\theta)}{2}\right)\left((\partial_2 v^1)^2+(\partial_2 v^2)^2\right)\\
		&\quad\quad +\sin(2\theta)(\partial_1 v^1\partial_2 v^1 +\partial_1 v^2\partial_2 v^2)\Bigg] d\theta d s\\
		&=\frac{1}{r^2}|\Grad v(x)|^2\int_0^r s^3 ds = \frac{r^2}{4}|\Grad v(x)|^2.
	\end{align*}
	Thus we obtain
	\begin{equation*}
		\frac{r^2}{4}\norm{\w}_{L^2}^2 = \frac{r^2}{4}\norm{\Grad v}_{L^2}^2\leq 2 S_2(v;r)^2+ C r^4  \norm{\Grad \w}_{L^2}^2,
	\end{equation*}
	which proves~\eqref{eq:ineqformega}. Note that this identity is independent of time and realization of the random event $\omega$. Thus after dividing by $r^2/4$, we integrate the above identity in time on $[\delta\wedge \tau,t\wedge \tau]$ and then on $[\delta\wedge \tau, T\wedge \tau]$ and take expectations to obtain
	\begin{equation}\label{eq:notsure}
		\begin{split}
	\E\int_{\delta\wedge\tau}^{T\wedge \tau}\int_{\delta\wedge\tau}^{t\wedge \tau}\norm{\w}_{L^2}^2 ds dt&\leq \frac{8}{r^2}\E\int_{\delta\wedge\tau}^{T\wedge \tau}\int_{\delta\wedge\tau}^{t\wedge \tau}S_2(v;r)^2ds dt + 4 C r^2 \E\int_{\delta\wedge\tau}^{T\wedge \tau}\int_{\delta\wedge\tau}^{t\wedge \tau} \norm{\Grad \w}_{L^2}^2 dsdt \\
	&\leq  \frac{C \phi(r)^2}{r^2} + C r^2 \E\int_{\delta\wedge\tau}^{T\wedge \tau}\int_{\delta\wedge\tau}^{t\wedge \tau} \norm{\Grad \w}_{L^2}^2 dsdt .				
		\end{split}
	\end{equation}
	We choose $r$ to balance terms, c.f.~\cite[p. 1094]{Lanthaler2021},
	\begin{equation*}
		r = \frac{\phi(\bar{r})^{1/2}}{\left(\E\int_{\delta\wedge\tau}^{T\wedge \tau}\int_{\delta\wedge\tau}^{t\wedge \tau} \norm{\Grad \w}_{L^2}^2 dsdt\right)^{1/4} }, \quad \text{where }\quad \bar{r}: = \frac{\bar{\phi}^{1/2}}{\left(\E\int_{\delta\wedge\tau}^{T\wedge \tau}\int_{\delta\wedge\tau}^{t\wedge \tau} \norm{\Grad \w}_{L^2}^2 dsdt \right)^{1/4}},
	\end{equation*}
	and $\bar{\phi}>0$ provides an upper bound $\phi(r)\leq \bar{\phi}$.  Then~\eqref{eq:notsure} can be bounded as (recall that $\phi$ is monotonically increasing)
	\begin{equation}\label{eq:prelim}
\E\int_{\delta\wedge\tau}^{T\wedge \tau}\int_{\delta\wedge\tau}^{t\wedge \tau} \norm{ \w}_{L^2}^2 dsdt \leq C \phi(\bar{r})\left(\E\int_{\delta\wedge\tau}^{T\wedge \tau}\int_{\delta\wedge\tau}^{t\wedge \tau} \norm{\Grad \w}_{L^2}^2 dsdt \right)^{1/2}.
	\end{equation}
	For simplicity, we denote $\widetilde{\phi}(r) = (C\phi(\bar{\phi}^{1/2}r))^2$ which is still a monotonically increasing function of $r$ and $\widetilde{\phi}(0)=0$, thus it is a modulus of continuity. Using this on the right hand side of~\eqref{eq:prelim} yields
	\begin{equation}\label{eq:prelim2}
			\left(\E\int_{\delta\wedge\tau}^{T\wedge \tau}\int_{\delta\wedge\tau}^{t\wedge \tau} \norm{ \w}_{L^2}^2 dsdt\right)^2 \leq \widetilde{\phi}\left(\left(\E\int_{\delta\wedge\tau}^{T\wedge \tau}\int_{\delta\wedge\tau}^{t\wedge \tau} \norm{\Grad \w}_{L^2}^2 dsdt\right)^{-1/4}\right)\E\int_{\delta\wedge\tau}^{T\wedge \tau}\int_{\delta\wedge\tau}^{t\wedge \tau} \norm{\Grad \w}_{L^2}^2 dsdt.
	\end{equation}
We define $f(z)= \widetilde{\phi}(z^{-1/4})z$.  Clearly, we have $f(z)=o(z)$ as $z\to\infty$, because $\widetilde{\phi}$ is a modulus of continuity, and goes to zero as $z\to\infty$. From~\cite[Lemma C.1]{Lanthaler2021}, we obtain that there is a dominating function $F(z)\geq f(z)$ still satisfying  $F(z)=o(z)$, which is continuous and stricty monotonically increasing and thus has an inverse $F^{-1}$ which grows superlinearly at infinity and can be represented in the form $F^{-1}(y)=\gamma(\sqrt{y})y$ where $\gamma:\R_{\geq 0}\to\R_{\geq 0}$ is a continuous monotonically increasing function with the property that there exists $\gamma_0>0$ such that $\gamma(\sqrt{y})\geq \gamma_0>0$ for all $y\geq 0$ and $\gamma(\sqrt{y})\to \infty$ as $y\to\infty$.  Thus, using this result, we obtain from~\eqref{eq:prelim2},
\begin{equation*}
		\left(\E\int_{\delta\wedge\tau}^{T\wedge \tau}\int_{\delta\wedge\tau}^{t\wedge \tau} \norm{ \w}_{L^2}^2 dsdt\right)^2 \gamma\left(\E\int_{\delta\wedge\tau}^{T\wedge \tau}\int_{\delta\wedge\tau}^{t\wedge \tau} \norm{\w}_{L^2}^2 dsdt\right) \leq\E\int_{\delta\wedge\tau}^{T\wedge \tau}\int_{\delta\wedge\tau}^{t\wedge \tau} \norm{\Grad \w}_{L^2}^2 dsdt.
\end{equation*}
	which proves the claim.
\end{proof}
Next, we prove a bound on the vorticity that is independent of the initial vorticity $\bar{\eta}_\nu=\Curl \bar{u}_\nu$ (which might blow with $\nu$.) The proof is an adaption of~\cite[Lemma 3.7]{Jin2024} to the stochastic setting.
\begin{lemma}
	\label{lem:vorticitybound}
	Let $u_\nu$ be a solution of the stochastic Navier-Stokes equations~\eqref{SNS} with forcing $f_\nu\in L^2(0,T;\HH)$ and initial data $\bar{u}_\nu\in L^2(\Omega;\HH)$ as in Definition~\ref{def:strongsolNS}. Let $\tau$ be a stopping time that is almost surely bounded and positive, i.e. $0<\tau\leq c$ a.s., and satisfies $\E t\wedge \tau>0$ for any $t>0$.  Then there exists a constant $C>0$ depending on $\norm{f_\nu}_{L^2(0,T;\HH)}$, $\E\norm{\bar{u}_\nu}_{L^2}^2$ and $\bar{\rho}$ and $T$  such that the vorticity $\w_\nu = \Curl u_\nu$ satisfies
	\begin{equation*}
	\E \left(T\wedge \tau \norm{\w_\nu(t\wedge \tau)}_{L^2}^2\right)\leq \frac{1}{\nu t} C(f,T,\bar{\sigma},\bar{\rho}).
	\end{equation*}
\end{lemma}
\begin{proof}
	Since $u_\nu$ is a solution of~\eqref{SNS}, we obtain from~\cite[Proposition 2.4.8]{Kuksin2012}, 
	that	 it satisfies an energy balance in mean, which we rewrite for any $t\in [0,T]$ as
		\begin{equation}\label{eq:EB}
		\begin{split}
			\nu \E \int_0^{t\wedge \tau} \norm{\w_\nu}_{L^2}^2 dt &=\frac12\E\norm{\bar{u}_\nu}_{L^2}^2 -\frac12\E\norm{u_\nu(t\wedge \tau)}_{L^2}^2 + \E\int_0^{t\wedge \tau} (f_\nu,u_\nu) dt + \frac12\bar{\sigma} \E t\wedge \tau\\
			& \leq\frac12\E\norm{\bar{u}_\nu}_{L^2}^2 -\frac12\E\norm{u_\nu(t\wedge \tau)}_{L^2}^2 + \frac12\norm{f_\nu}_{L^1(0,T;L^2(\dom))}^2+\frac12\sup_{s\in [0,t]}\E\norm{u_\nu(s)}_{L^2}^2 + \frac12\bar{\sigma} t\\
			&\leq C(\E\norm{\bar{u}_\nu}_{L^2}^2, \norm{f_\nu}_{L^1(0,T;L^2(\dom))},\bar{\sigma},T).
		\end{split}
	\end{equation}
Next, we consider the energy balance for $\w_\nu$: Since a priori, we can only expect $\w_\nu\in L^2(0,T;\HH)$ from the energy balance for $u_\nu$,~\eqref{eq:EB}, we consider a Galerkin approximation in which we will pass to the limit at the end. So let $\{e_i\}_{i=1}^\infty$ the previous basis of smooth Stokes eigenfunctions of $\HH$ used to define $\sigma\cdot  W$ and define as in~\cite[Theorem 3.1]{Flandoli1995} or~\cite[Proposition 2.1]{Bessaih1999} the Galerkin projection
\begin{equation}
	\label{eq:galerkinprojection}
	P_m v = \sum_{j=1}^m (v,e_i) e_i.
\end{equation}
Then, let us consider a Galerkin approximation of $u_\nu$ ($\w_\nu$ respectively)  denoted by $u^m_\nu$ ($\w_\nu^m$ respectively) where $m\in \N$ is the number of basis functions, i.e.,
\begin{equation*}
(d u_\nu^m , e_j)+\int_{\dom} \left((u^m_\nu \cdot \Grad) u^m_\nu\right) \cdot e_j dx dt +\nu \int_{\dom} \Grad u_\nu^m : \Grad e_j dx dt  = (f_\nu, e_j) dxdt + b_j d\beta_j(t), 
\end{equation*}
which translates to
\begin{multline}\label{eq:Galerkinforeta}
	(d \w_\nu^m ,\Curl  e_j)+\int_{\dom} \left((u^m_\nu \cdot \Grad) \w^m_\nu\right) \cdot \Curl e_j dx dt +\nu \int_{\dom} \Grad \w_\nu^m : \Grad \Curl e_j dx dt \\
	 = (\Curl f_\nu,\Curl  e_j) dt + b_j \norm{\Curl e_j}_{L^2}^2 d\beta_j(t), 
\end{multline}
for $\w_\nu^m$. We denote $\w_\nu^m(t,x) = \sum_{i=1}^m \w_i^m(t) \Curl e_i(x)$ and $f_\nu^m := P_m f_\nu$ as well as $\sigma^m dW: = P_m \sigma dW$. 
Next, we use It\^o's formula for $F(\w^m_\nu)=\norm{\w^m_\nu}_{L^2}^2$ (c.f.~\cite[Lemma 4.24, Theorem 4.32]{DaPrato2014})
\begin{align}\label{eq:energyvorticity}
	\begin{split}
	\norm{\w_\nu^m(t\wedge \tau)}_{L^2}^2 &= \norm{\w^m_\nu(r\wedge \tau)}_{L^2}^2 -2\nu \int_{r\wedge \tau}^{t\wedge \tau}\norm{\Grad \w^m_\nu}_{L^2}^2 ds + 2\int_{r\wedge \tau}^{t\wedge \tau}(\Curl f^m_\nu,\w_\nu^m) ds \\
	&\quad + \norm{\Curl \sigma^m}_{L^2}^2 (t\wedge \tau-r\wedge \tau) +2 \int_{r\wedge \tau}^{t\wedge \tau} \w_\nu^m \Curl \sigma^m dW(s)\\
	& \leq    \norm{\w^m_\nu(r\wedge \tau)}_{L^2}^2 -\nu \int_{r\wedge \tau}^{t\wedge \tau}\norm{\Grad \w_\nu^m}_{L^2}^2 ds + \frac{1}{\nu}\int_{r\wedge \tau}^{t\wedge \tau}\norm{f^m_\nu}_{L^2}^2 ds \\
	&\quad +\norm{\Curl \sigma^m}_{L^2}^2 t\wedge \tau +2 \int_{r\wedge \tau}^{t\wedge \tau} \w_\nu^m \Curl \sigma^m  dW(s).
\end{split}
\end{align}
We integrate this over $r\in [0,t\wedge \tau]$ and take expectations,
\begin{align}\label{eq:whatamess}
	\begin{split}
	\E \left(t\wedge \tau \norm{\w_\nu^m(t\wedge \tau)}_{L^2}^2\right) &\leq  \E\int_0^{t\wedge \tau }\norm{\w_\nu^m(r\wedge \tau)}_{L^2}^2dr  -\nu \E\int_0^{t\wedge \tau}\int_{r\wedge \tau}^{t\wedge \tau}\norm{\Grad \w_\nu^m}_{L^2}^2 ds dr + \frac{1}{\nu}\E\int_0^{t\wedge \tau}\int_{r\wedge \tau}^{t\wedge \tau}\norm{f^m_\nu}_{L^2}^2 ds dr \\
	&\quad + \norm{\Curl \sigma^m}_{L^2}^2 \E (t\wedge \tau)^2 +2\E\int_0^{t\wedge \tau} \int_{r\wedge \tau}^{t\wedge \tau} \w_\nu^m \Curl \sigma^m dW(s)dr \\
	& \leq \E\int_0^{t\wedge \tau }\norm{\w_\nu^m(r)}_{L^2}^2dr  + \frac{t}{\nu}\norm{f^m_\nu}_{L^2(0,T;\HH)}^2+\bar{\rho} t^2\\
	&\quad  +2\E\int_0^{t\wedge \tau}\!\! \int_{r\wedge \tau}^{t\wedge \tau} \w_\nu^m \Curl \sigma^m dW(s)dr.
	\end{split}
\end{align}
We need to estimate the last term, the stochastic integral. We can rewrite and bound it as follows, using first the Cauchy-Schwarz inequality and then It\^o isometry~\cite[Equation (4.30)]{DaPrato2014}:
\begin{align*}
	\left|\E\int_0^{t\wedge \tau}\!\! \int_{r\wedge \tau}^{t\wedge \tau} \w_\nu^m \Curl \sigma^m dW(s)dr\right|& = \left|\int_0^{t}\E\mathbf{1}_{[0,t\wedge \tau]}(r) \int_{r\wedge \tau}^{t\wedge \tau} \w_\nu^m \Curl \sigma^m dW(s)dr\right|\\
	& \leq \int_0^t \left(\E \mathbf{1}_{[0,t\wedge \tau]}(r)\right)^{1/2}\left(\E\left(\int_{0}^{t} \mathbf{1}_{[r\wedge \tau,t\wedge \tau]}(s)\w_\nu^m(s) \Curl \sigma^m dW(s)\right)^2\right)^{1/2} dr \\
	& \leq \int_0^t \left(\E \int_0^t \mathbf{1}_{[r\wedge \tau,t\wedge \tau]}(s)\sum_{k=1}^m b_k^2 | (\w_\nu^m,\Curl e_k)|^2 ds\right)^{1/2} dr \\
	& \leq \int_0^t \left(\E\int_{r\wedge \tau}^{t\wedge \tau}\norm{\w_\nu^m}_{L^2}^2 ds  \right)^{1/2}dr \left(\sum_{k=1}^\infty b_k^2 \norm{\Curl e_k}_{L^2}^2 \right)^{1/2}\\
	& \leq \bar{\rho}^{1/2} t \left(\E\int_{0}^{t\wedge \tau}\norm{\w_\nu^m}_{L^2}^2 ds  \right)^{1/2}\\
	& \leq \frac{\bar{\rho}^{1/2} t}{\nu} C(\E\norm{\bar{u}_\nu}_{L^2}^2, \norm{f_\nu}_{L^1(0,T;L^2(\dom))},\bar{\sigma},T),
\end{align*}
by~\eqref{eq:EB} (which can be repeated for the Galerkin approximation and yields the same, uniform in $m$ bound). Thus, we can estimate~\eqref{eq:whatamess} as follows:
\begin{equation*}
		\E \left(t\wedge \tau \norm{\w_\nu^m(t\wedge \tau)}_{L^2}^2\right)\leq \frac{1}{\nu } C(f,T,\bar{\sigma},\bar{\rho}).
\end{equation*}
Next, since  
\begin{equation*}
	\frac{t\wedge \tau}{T\wedge \tau}\geq \frac{t}{T},\quad \text{a.s.},
\end{equation*}
since $\tau>0$ a.s. and $0<t\leq T$, we can lower bound this by
\begin{equation*}
	\frac{t}{T}\E \left(T\wedge \tau \norm{\w_\nu^m(t\wedge \tau)}_{L^2}^2\right)\leq \E \left(t\wedge \tau \norm{\w_\nu^m(t\wedge \tau)}_{L^2}^2\right)\leq \frac{1}{\nu } C(f,T,\bar{\sigma},\bar{\rho}).
\end{equation*}
Thus, we obtain  after multiplying this inequality by $\frac{T}{t}$
\begin{equation*}
\E \left(T\wedge \tau \norm{\w_\nu^m(t\wedge \tau)}_{L^2}^2\right)\leq \frac{1}{\nu t} C(f,T,\bar{\sigma},\bar{\rho}).
\end{equation*}
Note that this bound is uniform in $m\in \N$. Thus, we can pass $m\to \infty$ and obtain the result for the limit $\w_\nu$. That $u^m_\nu \to u_\nu$ and $\w^m_\nu\to \w_\nu$ in a unique manner follows as in~\cite[Proposition 2.1]{Bessaih1999}.

\end{proof}
Now we are ready to prove Proposition~\ref{prop:tightnesstoenergybalance}.
\begin{proof}[Proof of Proposition~\ref{prop:tightnesstoenergybalance}]
		From~\eqref{eq:NSenergybalance}, we know that  solution to the stochastic Navier-Stokes equations satisfy an energy balance
	\begin{equation}\label{eq:nsenergybalance}
		\E\norm{u_{\nu}(t)}_{L^2}^2 = \E\norm{\bar{u}_\nu}_{L^2}^2 -2 \nu \E\int_0^t\norm{\Grad u_\nu}_{L^2}^2 + 2 \E \int_0^t (u_\nu,f_\nu)ds +\bar{\sigma}t,
	\end{equation}
	see e.g.~\cite[Proposition 2.4.8]{Kuksin2012}. As in~\cite{Flandoli1995} one can also show that
	\begin{equation}
		\label{eq:linfl2ubound}
		\E\sup_{t\in [0,T]}\norm{u_\nu}_{L^2}^p\leq C(p).
	\end{equation}
	(To make clear that this bound is independent of $\nu$, we state the proof in the appendix, c.f.~Lemma~\ref{lem:tsupL2boundu}.)

	In order to show tightness of the laws of $u_\nu$, we proceed as in~\cite[Theorem 3.1]{Flandoli1995}, we write $u_\nu$ as
	\begin{equation}\label{eq:approx}
		u_\nu(t) = \underbrace{\bar{u}_{\nu}}_{\text{I}}-\underbrace{\nu \int_0^t A u_\nu(s) ds}_{\text{II}} -\underbrace{\int_0^t  B(u_\nu(s),u_\nu(s)) ds}_{\text{III}} + \underbrace{\int_0^t f_\nu(s) ds}_{\text{IV}}+ \underbrace{\int_0^t \sigma dW(s)}_{\text{V}},
	\end{equation}
	where $A=-P\Delta: D(A)\to \HH$ is the Stokes operator ($P$ the Leray projector) and $B:\VV\times \VV\to (\VV)^*$ is the nonlinear operator defined through the trilinear form:
	\begin{equation*}
		b(u,v,w)=\int_{\dom} (u\cdot \Grad)v \cdot w dx, \quad \langle B(u,v),w\rangle = b(u,v,w).
	\end{equation*}
	$B$ can be extended to a continuous operator $B:\HH\times \HH\to \mathcal{D}(A^{-\kappa})$ for certain $\kappa>1$. The goal is to get a bound on the laws of $u_\nu$ in $W^{\beta,p}(0,T;(\VV\cap H^k(\dom))^*)$ for certain $\beta, p, k$. Thus we estimate
	\begin{equation*}
		\E\norm{\text{I}}_{L^2(\dom)}^2\leq C,
	\end{equation*}
	\begin{equation*}
		\E\norm{\text{II}}_{H^1(0,T;(\VV)^*)}^2\leq C,
	\end{equation*}
	by the energy balance~\eqref{eq:nsenergybalance}. For the fourth term, we have
	\begin{equation*}
		\E\norm{\text{IV}}_{H^1(0,T;L^2(\dom))}^2 \leq C,
	\end{equation*}
	due to the assumption that $f\in L^2([0,T];\HH)$. For the fifth term, Lemma 2.1 in~\cite{Flandoli1995} yields for all $\beta\in (0,1/2)$ and $p\geq 2$
	\begin{equation}\label{eq:fifthterm}
		\E\norm{\text{V}}_{W^{\beta,p}(0,T;\HH)}^p\leq C(\beta).
	\end{equation}
	Finally, for the third term, we use the skew-symmetry of the trilinear form $b$:
	\begin{equation*}
		\left|b(u,u,v)\right|=\left|b(u,v,u)\right|\leq \norm{u}_{L^2}^2 \norm{\Grad v}_{L^\infty}\leq \norm{u}_{L^2}^2 \norm{v}_{H^3\cap \VV},
	\end{equation*}	
	therefore 
	\begin{equation*}
		\norm{B(u_\nu(s),u_\nu(s))}_{L^2(0,T;(H^3\cap \VV)^*)}^2\leq C\norm{u_\nu}^4_{L^\infty(0,T;L^2(\dom))},
	\end{equation*}	
	hence
	\begin{equation*}
		\E\norm{\text{III}}_{H^1(0,T;(H^3\cap \VV)^*)}\leq C\E\norm{u_\nu}_{L^\infty(0,T;L^2(\dom))}^2\leq C,
	\end{equation*}
	using~\eqref{eq:linfl2ubound}. Pick $\beta$ and $p\geq 2$ in~\eqref{eq:fifthterm} such that $\beta p>1$ and denote $Z:= H^1(0,T;(H^3\cap \VV)^*)+ W^{\beta,p}(0,T;(H^3\cap \VV)^*)$. Then we have, collecting all the terms,
	\begin{equation*}
		\E\norm{u_\nu}_{Z}\leq C.
	\end{equation*}	
	This implies together with~\eqref{eq:weakstructurefcn2} and~\eqref{eq:linfl2ubound}
	\begin{equation*}
		\E\sup_{t\in[0,T]}\norm{u_\nu}_{L^2}^p+ \E\sup_{0<r<r_0}\frac{S_2^T(u_\nu,r)^2}{\phi(r)^2}+	\E\norm{u_\nu}_{Z}\leq C.
	\end{equation*}
	By Chebyshev inequality, we have for any $M>0$,
	\begin{equation*}
		\E\sup_{0<r<r_0}\frac{S_2^T(u_\nu,r)^2}{\phi(r)^2}\geq M \P\left[\sup_{0<r<r_0} \frac{S_2^T(u_\nu,r)^2}{\phi(r)^2}> M\right].
		\end{equation*}
Thus, we obtain for any $M>0$,
	\begin{multline*}
		\P\left[   	\sup_{t\in[0,T]}\norm{u_\nu}_{L^2}^p+ \sup_{r>0}\frac{ S_2^T(u_\nu,r)^2}{\phi(r)^2}+	\norm{u_\nu}_{Z} >M\right]\\
		\leq \frac{1}{M}\left(\E\sup_{t\in[0,T]}\norm{u_\nu}_{L^2}^p+ \E\sup_{0<r<r_0}\frac{ S_2^T(u_\nu,r)^2}{\phi(r)^2}+	\E\norm{u_\nu}_{Z}\right)\leq \frac{C(\beta,p)}{M},
	\end{multline*}
	thus for a given $\epsilon>0$, choosing $M_\epsilon$ such that $\frac{C(\beta,p)}{M_\epsilon}<\epsilon$, we get
	\begin{equation}\label{eq:tightness}
		\P\left[   	\sup_{t\in[0,T]}\norm{u_\nu}_{L^2}^p+ \sup_{r>0}\frac{S_2^T(u_\nu,r)^2}{\phi(r)^2}+	\norm{u_\nu}_{Z}\leq M_\epsilon\right]\geq 1-\epsilon.
	\end{equation}
	The set 
	\begin{equation*}
		K_\epsilon: = \left\{  u_\nu\in L^2(0,T;\HH)\,: \, \sup_{t\in[0,T]}\norm{u_\nu}_{L^2}^p+ \sup_{r>0}\frac{S_2^T(u_\nu,r)^2}{\phi(r)^2}+	\norm{u_\nu}_{Z}\leq M_\epsilon\right\},
	\end{equation*}
	is a compact set in $L^2(0,T;\HH)$ by Proposition~\ref{prop:moctimedep}, thus by Prokohorov's theorem, the sequence of measures $\mu_\nu = \P\circ u_\nu^{-1}$ is tight in $L^2(0,T;\HH)$.
Furthermore, since we obtain from Proposition~\ref{prop:moctimedep} that $Z\subset C^\alpha(0,T;(H^3\cap \VV)^*)$ for $\alpha = \min\{1/2,\beta-1/p\}$ and $C^\alpha([0,T];(H^3\cap\VV)^*)\cap L^\infty(0,T;\HH)\subset C([0,T];(\HH)_w)$ by~\cite[Lemma II.5.9]{Boyer2013}, the $\mu_\nu$ are bounded on $C([0,T];\HH_w)$ and thus tight in $L^2(0,T;\HH)\cap C([0,T];\HH_w)$ by~\cite[Lemma 5.1.12 (c)]{Ambrosio2008}, c.f. Remark~\ref{rem:weaktight}.
We pick a weakly convergent subsequence denoted by $\nu$ for ease of notation.
	Next, by the Skorokhod embedding theorem, there exists a stochastic basis $(\widetilde{\Omega},\widetilde{\mathcal{F}},\{\widetilde{\mathcal{G}}_t\}_{t\in [0,T]},\widetilde{\P})$ and on this basis $L^2(0,T;\HH)\cap C([0,T];(\HH)_w$-valued random variables $\widetilde{u}$, $\widetilde{u}_{\nu}$, $\nu>0$ such that $\widetilde{u}_\nu$ has the same law as $u_\nu$  on $L^2(0,T;\HH)\cap C([0,T];(\HH)_w$ and $\widetilde{u}_\nu\to\widetilde{u}$ in $L^2(0,T;\HH)\cap C([0,T];\HH_w)$ $\widetilde{\P}$-almost surely. Hence, we can pass to the limit in the weak formulation of all terms in~\eqref{eq:approx} to find that the limit satisfies for all $v\in \VV$ $\widetilde{\P}$-a.s.
	\begin{equation*}
		(\widetilde{u}(t),v)=(\widetilde{{u}}(0),v)+\int_0^t b(\widetilde{u},v,\widetilde{u}) ds+\int_0^t(f(s),v)ds+\int_0^t(\widetilde{\sigma},v)d\widetilde{W}(s).
	\end{equation*}
	Since $\widetilde{u}\in C([0,T];\HH_w)$, it is continuous at zero and we have almost surely $\widetilde{u}(0) = \widetilde{\bar{u}}=\bar{u}$ since the limit of $\{\bar{u}_\nu\}_{\nu>0}$ is unique.
	Moreover, since $\widetilde{W}$ is a Brownian motion on the probability space $(\widetilde{\Omega},\widetilde{\mathcal{F}},\{\widetilde{\mathcal{G}}_t\}_{t\in [0,T]},\widetilde{\P})$, $\widetilde{u}$ is a martingale solution. 
	
	It remains to show the energy balance. To do so, we introduce the following stopping time for $M>0$,
	\begin{equation}
		\label{eq:everybodystop}
		\tau^{\nu,M} = \inf\{ t>0\, : \, \int_0^t \norm{\widetilde{u}_\nu(s)}_{L^2}^2 ds \geq M\}\wedge T.
	\end{equation}
	This is a stopping time since it is adapted to the filtration $\{\widetilde{\mathcal{G}}_t\}_{t\in [0,T]}$ and the stochastic process $X^\nu(t):=\int_0^t \norm{\widetilde{u}_\nu(s)}_{L^2}^2 ds$ has a continuous modification by Kolmogorov's extension theorem. Indeed, we have for $s,t\geq 0$, and $\alpha>0$
	\begin{equation*}
		\widetilde{\E}|X^\nu(t)-X^\nu(s)|^\alpha = \widetilde{\E}\left|\int_s^t \norm{\widetilde{u}_\nu(r)}_{L^2}^2dr \right|^\alpha \leq |t-s|^\alpha\widetilde{\E}\sup_{r\in [0,T]}\norm{\widetilde{u}_\nu(r)}_{L^2}^{2\alpha} \leq C(2\alpha)|t-s|^\alpha,
	\end{equation*}
	for any $\frac{p}{2}>\alpha>0$ by~\eqref{eq:linfl2ubound}. Thus picking $\alpha>1$ (which is possible since w assume that $p>2$), we obtain that $X^\nu$ has a $(\alpha-1)$-H\"older continuous modification by the Kolmogorov extension theorem, for any $\nu\geq 0$. Since $M>0$ and $\int_0^t\norm{\widetilde{u}_\nu}_{L^2}^2 ds=0$ for $t=0$ a.s. and $\tau^{\nu,M}$ is a.s. continuous, we have that $\tau^{\nu,M}>0$ almost surely. We also note that this stopping time is bounded, $\tau^{\nu,M}\leq T$, by definition, and since $\widetilde{u}_\nu\to \widetilde{u}$ in $L^2(0,T;\HH)$, we have
	\begin{equation}\label{eq:heywhatsthatsound}
		\tau^M:= \inf\{t>0\,: \, \int_0^t \norm{\widetilde{u}(s)}_{L^2}^2 ds \geq M\}\wedge T \leq \liminf_{\nu\to 0} \tau^{\nu,M},\quad \text{a.s.}
	\end{equation}
	as for example in~\cite[Section 3]{Gyongy1996}. Now since we have from the previously established a.s. $L^2(0,T;\HH)$-convergence of $\widetilde{u}_\nu$, for any continuous $\theta\in C([0,T])$, and almost every $\omega\in\Omega$,
	\begin{equation*}
		\int_0^T\theta(t) \norm{\widetilde{u}_\nu(t)-\widetilde{u}(t)}_{L^2(\dom)}^2 dt\stackrel{\nu\to 0}{\longrightarrow} 0 ,
	\end{equation*}
	and therefore also for almost every $\omega\in \Omega$,
	\begin{equation*}
	\int_0^{T\wedge \tau^M}\theta(t) \norm{\widetilde{u}_\nu(t)-\widetilde{u}(t)}_{L^2(\dom)}^2 dt =\int_0^{T}\theta(t)\mathbf{1}_{[0,T\wedge \tau^M]}(t) \norm{\widetilde{u}_\nu(t)-\widetilde{u}(t)}_{L^2(\dom)}^2 dt \stackrel{\nu\to 0}{\longrightarrow} 0.
\end{equation*}
Since, by the definition of the stopping time and~\eqref{eq:heywhatsthatsound}, we have
 \begin{equation*}
 		\int_0^{T\wedge \tau^M}\theta(t) \norm{\widetilde{u}_\nu(t)-\widetilde{u}(t)}_{L^2(\dom)}^2 dt\leq 2 M,
 \end{equation*}
 for almost every $\omega\in \Omega$, we can use the dominated convergence theorem to conclude that
 \begin{equation}\label{eq:everybodylook}
 	\widetilde{\E}	\int_0^{T\wedge \tau^M}\theta(t) \norm{\widetilde{u}_\nu(t)-\widetilde{u}(t)}_{L^2(\dom)}^2 dt=\widetilde{\E}	\int_0^{T\wedge \tau^M}\theta(t) \norm{\widetilde{u}_\nu(t\wedge \tau^M)-\widetilde{u}(t\wedge \tau^M)}_{L^2(\dom)}^2 dt\stackrel{\nu\to 0}{\longrightarrow} 0,
 \end{equation}
 and therefore also
\begin{equation}\label{eq:lhs}
	\widetilde{\E}	\int_0^{T\wedge \tau^M}\theta(t) \norm{\widetilde{u}_\nu(t)}_{L^2(\dom)}^2 dt\stackrel{\nu\to 0}{\longrightarrow}\widetilde{\E}	\int_0^{T\wedge \tau^M}\theta(t) \norm{\widetilde{u}(t)}_{L^2(\dom)}^2 dt.
\end{equation}
From~\eqref{eq:everybodylook}, we also obtain that for almost every $t\in [0,T]$,
\begin{equation}
	\label{eq:whatsgoingdown}
	\widetilde{\E} \norm{\widetilde{u}_\nu(t\wedge \tau^M)-\widetilde{u}(t\wedge \tau^M)}_{L^2(\dom)}^2 \stackrel{\nu\to 0}{\longrightarrow} 0.
\end{equation}
Next, we observe that the limit $\widetilde{u}(t)$ is continuous near zero in the sense that $\widetilde{\E}\norm{\widetilde{u}(t)-\widetilde{\bar{u}}}_{L^2}^2\to 0$ as $t\to 0$. To see this, we note first since $\widetilde{u}\in C([0,T];\HH_w)$, we have that $t\mapsto (\widetilde{u}(t),v)$ is continuous for $v\in L^2(\Omega;\HH)$ that is measurable with respect to $\widetilde{\mathcal{G}}_0$ for almost every $\omega\in \Omega$. Then since $|(\widetilde{u}(t),v)|\leq \sup_{t\in [0,T]}\norm{\widetilde{u}(t)}_{L^2}\norm{v}_{L^2}$, we can use~\eqref{eq:linfl2ubound} combined with the dominated convergence theorem to conclude that
\begin{equation}\label{eq:weakcont}
	t\mapsto \widetilde{\E}(\widetilde{u}(t),v),\quad \text{is continuous. }
\end{equation} 
Then, by Fatou's lemma, 
\begin{equation}\label{eq:liminf}
{\E}\norm{{\bar{u}}}_{L^2}^2 =\widetilde{\E}\norm{\widetilde{\bar{u}}}_{L^2}^2\leq \liminf_{t\to 0}\widetilde{\E}\norm{\widetilde{u}(t\wedge \tau^M)}_{L^2}^2.
\end{equation}
On the other hand, we obtain from the energy balance~\eqref{eq:nsenergybalance} (c.f. Equation (2.121) in~\cite{Kuksin2012})
\begin{equation}\label{eq:energyineqnu}
	\widetilde{\E}\norm{\widetilde{u}_{\nu}(t\wedge \tau^M)}_{L^2}^2 \leq  \E\norm{\bar{u}_\nu}_{L^2}^2  + 2 \widetilde{\E} \int_0^{t\wedge\tau^M} (\widetilde{u}_\nu,f_\nu)ds +\bar{\sigma}\E t\wedge \tau^M.
\end{equation}
We can use the strong convergence of the initial data $\bar{u}_\nu \to \bar{u}$ in $L^2(\Omega;\HH)$, the weak convergence of $f_\nu$ in $L^2(0,T;\HH)$ combined with the strong convergence of $\widetilde{u}_\nu$ in $L^2(0,T;\HH)$ to pass to the limit in the terms on the right hand side of~\eqref{eq:energyineqnu}. On the left hand side, we use Fatou's lemma to obtain:
\begin{equation}
	\label{eq:eulerenergyinequality}
	\widetilde{\E}\norm{\widetilde{u}(t\wedge \tau^M)}_{L^2}^2\leq {\E}\norm{{\bar{u}}}_{L^2}^2+2\int_0^{t\wedge\tau^M}\widetilde{\E}(f,\widetilde{u}) ds +\bar{\sigma} \E t\wedge\tau^M. 
\end{equation}
Now using again~\eqref{eq:linfl2ubound} and that $f\in L^2(0,T;\HH)$, we can take the limsup over $t\to 0$ in this identity to obtain
\begin{equation*}
	\limsup_{t\to 0}	\widetilde{\E}\norm{\widetilde{u}(t\wedge \tau^M)}_{L^2}^2\leq {\E}\norm{{\bar{u}}}_{L^2}^2.
\end{equation*}
Combining this with~\eqref{eq:weakcont} and~\eqref{eq:liminf}, we obtain
\begin{equation}\label{eq:strongcontatzeroa}
	\lim_{t\to 0}\widetilde{\E}\norm{\widetilde{u}(t\wedge \tau^M)-\widetilde{\bar{u}}}_{L^2}^2 = 0, 
\end{equation} 
as claimed.
Now, let $\min(T,1)>\epsilon>0$ be an arbitrary small number. Then~\eqref{eq:strongcontatzeroa} implies that we can find $\delta_1>0$ such that $\delta_1^{1/2}<\epsilon$ and for any $t\in [0,\delta_1]$, we have
\begin{equation}
	\label{eq:initialterm}
	\widetilde{\E}\left|\norm{\widetilde{u}(t\wedge\tau^M)}_{L^2}^2 - \norm{\widetilde{\bar{u}}}_{L^2}^2\right|<\epsilon.
\end{equation}
Next, let $0<\delta_2\leq \delta_1$ be such that for any $0<\delta\leq \delta_2$,
\begin{equation}\label{eq:fdelta}
	\int_0^\delta \norm{f(s)}_{L^2}^2 ds <\epsilon.
\end{equation}
This is possible since $f\in L^2(0,T;\HH)$. Now we pick $0<\delta\leq \delta_2$ for which~\eqref{eq:whatsgoingdown} holds.
We recall the energy balance Equation (2.121) in~\cite{Kuksin2012}, i.e., we have for a.e. $\omega\in\Omega$,
\begin{multline*}
	\norm{\widetilde{u}_\nu(t\wedge \tau^M)}_{L^2(\dom)}^2 = \norm{\widetilde{\bar{u}}_\nu}_{L^2(\dom)}^2  +2 \int_{0}^{t \wedge \tau^M}(f_\nu,\widetilde{u}_\nu) ds-2\nu\int_{0}^{t\wedge \tau^M} \norm{\Grad \widetilde{u}_\nu}_{L^2(\dom)}^2 ds\\
	+\bar{\sigma} t\wedge \tau^M +2\int_{0}^{t\wedge \tau^M} \widetilde{u}_\nu \sigma d\widetilde{W}(s).
\end{multline*}
 Subtracting this identity for $t=\delta$ from the energy balance for $t$, we have for any $\nu>0$, a.s.
 \begin{multline}\label{eq:aeenergybalance}
 	\norm{\widetilde{u}_\nu(t\wedge \tau^M)}_{L^2(\dom)}^2 = \norm{\widetilde{u}_\nu(\delta\wedge \tau^M)}_{L^2(\dom)}^2  +2 \int_{\delta\wedge \tau^M}^{t \wedge \tau^M}(f_\nu,\widetilde{u}_\nu) ds-2\nu\int_{\delta\wedge \tau^M}^{t\wedge \tau^M} \norm{\Grad \widetilde{u}_\nu}_{L^2(\dom)}^2 ds\\
 	 +\bar{\sigma} \left[t\wedge \tau^M -\delta \wedge \tau^M\right]+2\int_{\delta\wedge \tau^M}^{t\wedge \tau^M} \widetilde{u}_\nu \sigma d\widetilde{W}(s).
 \end{multline}

 We have a different stopping time than in~\cite{Kuksin2012}, however, in order to apply \cite[Theorem 7.7.5]{Kuksin2012}, we could take the minimum of our stopping time and the one in~\cite{Kuksin2012}, let's call it $\hat{\tau}^N$, which yields a new stopping time $\tau^M\wedge \hat{\tau}^N$, for which we can apply~\cite[Theorem 7.7.5]{Kuksin2012} and then pass to the limit $N\to \infty$ in $\hat{\tau}^N$ to obtain~\eqref{eq:aeenergybalance}.
  Multiplying~\eqref{eq:aeenergybalance} by $\theta$ and integrating this identity in time over $[\delta\wedge\tau^M,T\wedge \tau^M]$ and taking expectations, we obtain
 \begin{align}
 	\label{eq:doweneedalabel}
 	\begin{split}
 	&\widetilde{\E}\int_0^T\theta(t)\mathbf{1}_{[0,T\wedge\tau^M]}\norm{\widetilde{u}_\nu(t\wedge \tau^M)}_{L^2(\dom)}^2 dt\\
 	 &\quad  = 	\widetilde{\E}\int_{\delta\wedge\tau^M}^{T\wedge\tau^M}\theta(t)\norm{\widetilde{u}_\nu(t\wedge \tau^M)}_{L^2(\dom)}^2 dt +\widetilde{\E}\int_0^{\delta\wedge\tau^M}\theta(t)\norm{\widetilde{u}_\nu(t\wedge \tau^M)}_{L^2(\dom)}^2 dt\\
 	& \quad =\underbrace{\widetilde{\E}\int_0^T\theta(t)\mathbf{1}_{[0,\delta\wedge \tau^M]}\norm{\widetilde{u}_\nu(t\wedge \tau^M)}_{L^2(\dom)}^2 dt}_{\text{I}} + \underbrace{\widetilde{\E}\int_0^T\theta(t)\mathbf{1}_{[\delta\wedge \tau^M,T\wedge\tau^M]}\norm{\widetilde{u}_\nu(\delta\wedge \tau^M)}_{L^2(\dom)}^2 dt }_{\text{II}}
 	\\
 	&\qquad +  2\underbrace{\widetilde{\E}\int_0^T\theta(t)\mathbf{1}_{[\delta\wedge \tau^M,T\wedge\tau^M]}\int_{\delta\wedge \tau^M}^{t \wedge \tau^M}(f_\nu,\widetilde{u}_\nu) ds dt }_{\text{III}} + \underbrace{\widetilde{\E}\int_0^T\theta(t)\mathbf{1}_{[\delta\wedge \tau^M,T\wedge\tau^M]} \bar{\sigma} \left[t\wedge \tau^M -\delta \wedge \tau^M\right] dt }_{\text{IV}} 
 	 \\
 	& \qquad  +   2\underbrace{\widetilde{\E}\int_0^T\theta(t)\mathbf{1}_{[\delta\wedge \tau^M,T\wedge\tau^M]} \int_{\delta\wedge \tau^M}^{t\wedge \tau^M} \widetilde{u}_\nu \sigma d\widetilde{W}(s) 		 dt }_{\text{V}}+2\underbrace{\nu\widetilde{\E}\int_0^T\theta(t)\mathbf{1}_{[\delta\wedge \tau^M,T\wedge\tau^M]}\int_{\delta\wedge \tau^M}^{t\wedge \tau^M} \norm{\Grad \widetilde{u}_\nu}_{L^2(\dom)}^2 ds dt }_{\text{VI}}.
 	\end{split}
 	\end{align}
 Next, we compute the limits of all these terms as $\nu\to 0$. We already know that the left hand side converges from~\eqref{eq:lhs}. For term I, we have
 \begin{align*}
 	|\text{I}|\leq \norm{\theta}_{L^\infty}\delta \sup_{t\in [0,T]}\widetilde{\E}\norm{\widetilde{u}_\nu(t)}_{L^2}^2\leq C \delta\leq C \epsilon.
 \end{align*}
	For term II, since we are picking $\delta$ for which~\eqref{eq:whatsgoingdown} holds, and compute
	\begin{align*}
		\text{II} &= \widetilde{\E}\int_0^T\theta(t)\mathbf{1}_{[\delta\wedge \tau^M,T\wedge\tau^M]}\norm{\widetilde{u}(\delta\wedge \tau^M)}_{L^2(\dom)}^2 dt \\
		&\quad-\underbrace{\widetilde{\E}\int_0^T\theta(t)\mathbf{1}_{[\delta\wedge \tau^M,T\wedge\tau^M]}\left(\norm{\widetilde{u}(\delta\wedge \tau^M)}_{L^2(\dom)}^2 - \norm{\widetilde{u}_\nu(\delta\wedge \tau^M)}_{L^2(\dom)}^2\right)dt }_{\text{IIa}}.
	\end{align*}
	We bound the second term as follows:
	\begin{align*}
		|\text{IIa}|&\leq \widetilde{\E}\int_0^T|\theta(t)|\mathbf{1}_{[\delta\wedge \tau^M,T\wedge\tau^M]}\left|\norm{\widetilde{u}(\delta\wedge \tau^M)}_{L^2(\dom)}^2 - \norm{\widetilde{u}_\nu(\delta\wedge \tau^M)}_{L^2(\dom)}^2\right|dt\\
		& \leq \widetilde{\E}\left|\norm{\widetilde{u}(\delta\wedge \tau^M)}_{L^2(\dom)}^2 - \norm{\widetilde{u}_\nu(\delta\wedge \tau^M)}_{L^2(\dom)}^2\right| T \norm{\theta}_{L^\infty([0,T])},
	\end{align*}
	which goes to zero by~\eqref{eq:whatsgoingdown}.
	Thus, we have
	\begin{equation*}
		\text{II}\stackrel{\nu\to 0}{\longrightarrow}  \widetilde{\E}\int_0^T\theta(t)\mathbf{1}_{[\delta\wedge \tau^M,T\wedge\tau^M]}\norm{\widetilde{u}(\delta\wedge \tau^M)}_{L^2(\dom)}^2 dt
	\end{equation*}
	Now, from~\eqref{eq:initialterm}, we obtain
	\begin{equation*}
		\left| \widetilde{\E}\int_0^T\theta(t)\mathbf{1}_{[\delta\wedge \tau^M,T\wedge\tau^M]}\norm{\widetilde{u}(\delta\wedge \tau^M)}_{L^2(\dom)}^2 dt -  \widetilde{\E}\int_0^T\theta(t)\mathbf{1}_{[\delta\wedge \tau^M,T\wedge\tau^M]}\norm{\widetilde{\bar{u}}}_{L^2(\dom)}^2 dt\right| \leq \epsilon \norm{\theta}_{L^\infty} T,
	\end{equation*}
	thus
	\begin{equation*}
		\lim_{\nu\to 0}\left|\text{II}- \widetilde{\E}\int_0^{T\wedge \tau^M}\theta(t)\norm{\widetilde{\bar{u}}}_{L^2(\dom)}^2 dt\right| \leq C\epsilon.
	\end{equation*}
	Next, we rewrite term III as follows:
	\begin{align*}
		\text{III} &= \widetilde{\E}\int_0^T\theta(t)\mathbf{1}_{[\delta\wedge \tau^M,T\wedge\tau^M]}\int_{\delta\wedge \tau^M}^{t \wedge \tau^M}(f,\widetilde{u}) ds dt\\
		& \quad   +\underbrace{\widetilde{\E}\int_0^T\theta(t)\mathbf{1}_{[\delta\wedge \tau^M,T\wedge\tau^M]}\int_{\delta\wedge \tau^M}^{t \wedge \tau^M}(f_\nu-f,\widetilde{u}) ds dt }_{\text{IIIa}} + \underbrace{\widetilde{\E}\int_0^T\theta(t)\mathbf{1}_{[\delta\wedge \tau^M,T\wedge\tau^M]}\int_{\delta\wedge \tau^M}^{t \wedge \tau^M}(f_\nu,\widetilde{u}_\nu-\widetilde{u}) ds dt }_{\text{IIIb}}.
	\end{align*}
	For the second term on the right hand side we use the weak convergence of $f_\nu$ in $L^2(0,T;\HH)$ to obtain
	\begin{equation*}
		\text{IIIa}\stackrel{\nu\to 0}{\longrightarrow} 0.
	\end{equation*}
	For the third term on the right hand side, we use the uniform boundedness of the $f_\nu$ in $L^2(0,T;\HH)$ combined with~\eqref{eq:everybodylook} to obtain
	\begin{equation*}
		|\text{IIIb}|\leq \norm{\theta}_{L^\infty}T\norm{f_\nu}_{L^2(0,T;\HH)}\left(\widetilde{\E}\int_0^{T\wedge \tau^M}\norm{\widetilde{u}_\nu(t) -\widetilde{u}(t)}_{L^2}^2 dt\right)^{1/2}\stackrel{\nu\to 0}{\longrightarrow} 0.
	\end{equation*}

	Thus we have
	\begin{equation*}
		\text{III}\stackrel{\nu\to 0}{\longrightarrow}  \widetilde{\E}\int_0^T\theta(t)\mathbf{1}_{[\delta\wedge \tau^M,T\wedge\tau^M]}\int_{\delta\wedge \tau^M}^{t \wedge \tau^M}(f,\widetilde{u}) ds dt.
	\end{equation*}
	Now since 
	\begin{align*}
		&\left| \widetilde{\E}\int_{\delta\wedge \tau^M}^{T\wedge \tau^M}\!\!\!\theta(t)\int_{\delta\wedge \tau^M}^{t \wedge \tau^M}(f,u) ds dt - \widetilde{\E}\int_0^{T\wedge \tau^M}\!\!\!\theta(t)\int_{0}^{t \wedge \tau^M}(f,u) ds dt\right| \\
		&\quad \leq \left| \widetilde{\E}\int_0^{T\wedge \tau^M}\!\!\!\theta(t)\int_0^{\delta\wedge \tau^M}(f,u) ds dt \right| + \left| \widetilde{\E}\int_0^{\delta\wedge \tau^M}\!\!\!\theta(t)\int_0^{t\wedge \tau^M}(f,u) ds dt \right|\\
		&\quad \leq \norm{\theta}_{L^\infty} T \widetilde{\E}\sup_{t\in [0,T]}\norm{\widetilde{u}(t)}_{L^2}^2\int_0^\delta \norm{f(s)}_{L^2}^2 ds +  \norm{\theta}_{L^\infty} \delta \widetilde{\E}\sup_{t\in [0,T]}\norm{\widetilde{u}(t)}_{L^2}^2\int_0^\delta \norm{f(s)}_{L^2}^2 ds \\
		& \leq 2\epsilon \norm{\theta}_{L^\infty} T \widetilde{\E}\sup_{t\in [0,T]}\norm{\widetilde{u}(t)}_{L^2}^2,
			\end{align*}
			by~\eqref{eq:fdelta}, we have
			\begin{equation*}
				\lim_{\nu\to 0}\left|\text{III} - \widetilde{\E}\int_0^{T\wedge \tau^M}\!\!\!\theta(t)\int_{0}^{t \wedge \tau^M}(f,u) ds dt\right| \leq C \epsilon.
			\end{equation*}
	Term IV does not depend on $\nu$, we can estimate it in terms of $\delta$ and $\epsilon$ as follows:
	\begin{equation*}
		\left|\text{IV} - \widetilde{\E}\int_0^{T\wedge \tau^M}\theta(t) \bar{\sigma} t\wedge \tau^M  dt\right| \leq  C \norm{\theta}_{L^\infty}\bar{\sigma} T \delta\leq C \epsilon
	\end{equation*}
	We proceed to term V: We write it as
	\begin{equation*}
		\begin{split}
			\text{V}&= \widetilde{\E}\int_0^T\theta(t)\mathbf{1}_{[\delta\wedge \tau^M,T\wedge\tau^M]} \int_{\delta\wedge \tau^M}^{t\wedge \tau^M} \widetilde{u} \sigma d\widetilde{W}(s) 		 dt\\
			&+\underbrace{\int_0^T\theta(t)\widetilde{\E}\int_0^T\mathbf{1}_{[\delta\wedge \tau^M,T\wedge\tau^M]}(t) \mathbf{1}_{[\delta\wedge \tau^M,t\wedge\tau^M]}(s)
			(\widetilde{u}_\nu-\widetilde{u}) \cdot \sigma d\widetilde{W}(s) 	dt}_{\text{Va}}.
		\end{split}
	\end{equation*}
	We use the It\^o isometry for $\Phi^\nu(s):= \mathbf{1}_{[\delta\wedge \tau^M,t\wedge\tau^M]}(s)
	(\widetilde{u}_\nu-\widetilde{u})(s)$, c.f.~\cite[Equation (4.30)]{DaPrato2014} to show that term $\text{Va}$ goes to zero. We estimate, first using the Cauchy-Schwarz inequality and then the It\^o isometry:
	
	\begin{align*}
		|\text{Va}|&\leq  \int_0^T|\theta(t)|\left(\widetilde{\E}\mathbf{1}_{[\delta\wedge\tau^M,T\wedge \tau^M]}(t)\right)^{1/2}\left(\widetilde{\E}\left(\int_0^T\mathbf{1}_{[\delta\wedge\tau^M,t\wedge\tau^M]}(s)(\widetilde{u}_\nu-\widetilde{u})(s)\sigma d\widetilde{W}(s)\right)^2 \right)^{1/2} dt\\
		& \leq \int_0^T|\theta(t)|\left(\widetilde{\E}\left(\int_0^T\mathbf{1}_{[\delta\wedge\tau^M,t\wedge\tau^M]}(s)(\widetilde{u}_\nu-\widetilde{u})(s)\sigma d\widetilde{W}(s)\right)^2 \right)^{1/2} dt\\
		& = \int_0^T|\theta(t)|\left(\widetilde{\E}\left(\int_0^T\Phi^\nu(s)\sigma d\widetilde{W}(s)\right)^2 \right)^{1/2} dt\\
		&  = \int_0^T |\theta(t)| \left(\widetilde{\E}\int_0^T \sum_{k=1}^\infty b_k^2 \left| (\Phi^\nu(s),e_k)\right|^2 ds\right)^{1/2}dt\\
		& = \int_0^T|\theta(t)| \left(  \widetilde{\E}\int_0^T \mathbf{1}_{[\delta\wedge \tau^M,t\wedge\tau^M]}(s)\sum_{k=1}^\infty b_k^2 \left|( 
		(\widetilde{u}_\nu-\widetilde{u})(s),e_k)\right|^2 ds \right)^{1/2}dt\\
		& = \int_0^T |\theta(t)| \left( \widetilde{\E}\int_{\delta\wedge \tau^M}^{t\wedge\tau^M} \sum_{k=1}^\infty b_k^2 \left|( 		(\widetilde{u}_\nu-\widetilde{u})(s),e_k)\right|^2 ds\right)^{1/2} dt\\
		& \leq \int_0^T |\theta(t)|    \bar{\sigma}^{1/2} \left(\widetilde{\E}\int_0^{t\wedge \tau^M}\norm{\widetilde{u}(s)-\widetilde{u}_\nu(s)}_{L^2(\dom)}^2 ds \right)^{1/2}dt \stackrel{\nu\to 0}{\longrightarrow} 0.	
	\end{align*}
	by~\eqref{eq:everybodylook}. 
Hence, $\text{Va}\to 0$ as $\nu\to 0$ and
	\begin{equation*}
			\text{V}\stackrel{\nu\to 0}{\longrightarrow} \widetilde{\E}\int_0^T\theta(t)\mathbf{1}_{[\delta\wedge \tau^M,T\wedge\tau^M]} \int_{\delta\wedge \tau^M}^{t\wedge \tau^M} \widetilde{u} \sigma d\widetilde{W}(s) 		 dt.
	\end{equation*}
	We have
	\begin{align*}
		&\left|\widetilde{\E}\int_0^T\theta(t)\mathbf{1}_{[\delta\wedge \tau^M,T\wedge\tau^M]} \int_{\delta\wedge \tau^M}^{t\wedge \tau^M} \widetilde{u} \sigma d\widetilde{W}(s) 		 dt - \widetilde{\E}\int_0^T\theta(t)\mathbf{1}_{[0,T\wedge\tau^M]} \int_{0}^{t\wedge \tau^M} \widetilde{u} \sigma d\widetilde{W}(s) 		 dt\right|\\
		&\leq \left|\widetilde{\E}\int_{\delta\wedge \tau^M}^{T\wedge \tau^M}\theta(t)\int_{0}^{\delta\wedge \tau^M} \widetilde{u} \sigma d\widetilde{W}(s) 		 dt\right| + \left|\widetilde{\E}\int_0^{\delta\wedge \tau^M}\theta(t) \int_{0}^{t\wedge \tau^M} \widetilde{u} \sigma d\widetilde{W}(s) 		 dt\right|
	\end{align*}
	We use the It\^o isometry once more to show that the two terms on the right hand side are of order $\epsilon$:
	\begin{align*}
		 \left|\widetilde{\E}\int_{\delta\wedge \tau^M}^{T\wedge \tau^M}\theta(t)\int_{0}^{\delta\wedge \tau^M} \widetilde{u} \sigma d\widetilde{W}(s) 		 dt\right| & \leq \int_0^T |\theta(t)|\sqrt{\bar{\sigma}}\left(\widetilde{\E}\int_0^{\delta\wedge \tau^M} \norm{\widetilde{u}}_{L^2}^2 ds\right)^{1/2}dt\\
		 & \leq T \norm{\theta}_{L^\infty}\sqrt{\bar{\sigma}}\delta^{1/2}\left(\widetilde{\E}\sup_{s\in[0,T]} \norm{\widetilde{u}(s)}_{L^2}^2 \right)^{1/2}\\
		 & \leq C \epsilon,
	\end{align*}
	by the assumptions on $\delta$ and~\eqref{eq:linfl2ubound}. Similarly, the second term above can be show to be of order $\mathcal{O}(\epsilon)$. Thus we have that
	\begin{equation*}
		\lim_{\nu\to 0} \left|\text{V}-\widetilde{\E}\int_0^{T\wedge \tau^M}\theta(t)\int_{0}^{t\wedge \tau^M} \widetilde{u} \sigma d\widetilde{W}(s) 		 dt \right|\leq C \epsilon.
	\end{equation*}
	
	Next, we show that term $\text{VI}\to 0$. First, we note that we can upper bound this term as follows:
	\begin{align*}
		|\text{VI}|& = \left|\nu\widetilde{\E}\int_0^T\theta(t)\mathbf{1}_{[\delta\wedge \tau^M,T\wedge\tau^M]}\int_{\delta\wedge \tau^M}^{t\wedge \tau^M} \norm{\Grad \widetilde{u}_\nu}_{L^2(\dom)}^2 ds dt\right| \\
		& \leq \norm{\theta}_{L^\infty}\nu\widetilde{\E}\int_{\delta\wedge \tau^M}^{T\wedge \tau^M}\int_{\delta\wedge \tau^M}^{t\wedge \tau^M} \norm{\Grad \widetilde{u}_\nu}_{L^2(\dom)}^2 ds dt.		
	\end{align*}
	Hence it is sufficient to show that
	\begin{equation*}
		\text{VIa}: = \nu\widetilde{\E}\int_{\delta\wedge \tau^M}^{T\wedge \tau^M}\int_{\delta\wedge \tau^M}^{t\wedge \tau^M} \norm{\Grad \widetilde{u}_\nu}_{L^2(\dom)}^2 ds dt \stackrel{\nu\to 0}{\longrightarrow} 0 .		
	\end{equation*}
Without loss of generality, let us assume that the $\widetilde{u}_\nu$ are sufficiently smooth in space for the following calculations. If it is not the case, we can consider a Galerkin approximation as in Lemma~\ref{lem:vorticitybound}, c.f.,~\eqref{eq:galerkinprojection} and~\eqref{eq:Galerkinforeta}, and pass to the limit after establishing a uniform bound. We consider the energy balance for the vorticity $\w_\nu:=\widetilde{u}_\nu$, c.f.,~\eqref{eq:energyvorticity},
 using~\cite[Theorem 7.7.5]{Kuksin2012} for $F(\w_\nu)= \norm{\w_\nu}_{L^2}^2$ for any $\delta>0$,
	\begin{multline*}
		\norm{\w_\nu(t\wedge \tau^M)}^2_{L^2} = \norm{{\w}_\nu(\delta\wedge\tau^M)}^2_{L^2} - 2\nu\int_{\delta\wedge \tau^M}^{t\wedge\tau^M} \norm{\nabla \w_\nu}^2_{L^2}  ds\\
		- 2\int_{\delta\wedge \tau^M}^{t\wedge\tau^M} \ev{f_\nu, \curl \w_\nu} \, ds
		+ 2\int_{\delta\wedge \tau^M}^{t\wedge\tau^M} \ev{\curl(\sigma), \w_\nu} d\widetilde{W}(s) + \bar{\rho}(t\wedge\tau^M-\delta\wedge\tau^M).
	\end{multline*}
	With Young's inequality, we estimate
	$$- 2 \int_{\delta\wedge \tau^M}^{t\wedge\tau^M}  \ev{f_\nu, \curl \w_\nu} \, ds \le \nu \int_{\delta\wedge \tau^M}^{t\wedge\tau^M}  \norm{\nabla \w_\nu}^2_{L^2} \, ds + \nu^{-1} \int_{\delta\wedge \tau^M}^{t\wedge\tau^M}  \norm{f_\nu}^2_{L^2} \, ds.$$
	Integrating this identity on $[\delta\wedge \tau^M,T\wedge \tau^M]$ and taking expectations, we obtain
	\begin{align*}
		\widetilde{\E}\int_{\delta\wedge \tau^M}^{T\wedge \tau^M}\norm{\w_\nu(t\wedge \tau^M)}^2_{L^2}dt &\le \widetilde{\E}(T\wedge\tau^M-\delta\wedge\tau^M)\norm{{\w}_\nu(\delta\wedge\tau^M)}^2_{L^2} - \nu\widetilde{\E}\int_{\delta\wedge \tau^M}^{T\wedge \tau^M}\!\!\int_{\delta\wedge \tau^M}^{t\wedge\tau^M} \norm{\nabla \w_\nu}^2_{L^2} \, dsdt  \\
		&\quad + \nu^{-1} \widetilde{\E}\int_{\delta\wedge \tau^M}^{T\wedge \tau^M}\!\!\int_{\delta\wedge \tau^M}^{t\wedge\tau^M} \norm{f_\nu}^2_{L^2} \, dsdt  + \bar{\rho}\widetilde{\E} (T\wedge\tau^M-\delta\wedge\tau^M)^2\\
		& \leq \widetilde{\E}\left(T\wedge\tau^M\norm{{\w}_\nu(\delta\wedge\tau^M)}^2_{L^2}\right) - \nu\widetilde{\E}\int_{\delta\wedge \tau^M}^{T\wedge \tau^M}\!\!\int_{\delta\wedge \tau^M}^{t\wedge\tau^M} \norm{\nabla \w_\nu}^2_{L^2} \, dsdt  \\
		&\quad + \nu^{-1}T  \norm{f_\nu}^2_{L^2(0,T;\HH)}  + \bar{\rho}T^2.
	\end{align*}	
	Using Lemma~\ref{lem:vorticitybound} with $t=\delta$, we can upper bound the first term on the right hand side so that we obtain
	\begin{equation*}
		\widetilde{\E}\int_{\delta\wedge \tau^M}^{T\wedge \tau^M}\norm{\w_\nu(t\wedge \tau^M)}^2_{L^2}dt \le
	\frac{C}{\delta \nu} - \nu\widetilde{\E}\int_{\delta\wedge \tau^M}^{T\wedge \tau^M}\!\!\int_{\delta\wedge \tau^M}^{t\wedge\tau^M} \norm{\nabla \w_\nu}^2_{L^2} \, dsdt  + \nu^{-1}T  \norm{f_\nu}^2_{L^2(0,T;\HH)}  + \bar{\rho}T^2.
	\end{equation*}

	Now, we use the enhanced Poincar\'e Lemma~\ref{lem:poincaremean}, to bound this by
	\begin{multline*}
			\widetilde{\E}\int_{\delta\wedge \tau^M}^{T\wedge \tau^M}\norm{\w_\nu(t\wedge \tau^M)}^2_{L^2}dt \le \frac{C}{\delta \nu}\\
			-\nu\gamma\left(\widetilde{\E}\int_{\delta\wedge \tau^M}^{T\wedge \tau^M}\!\!\int_{\delta\wedge \tau^M}^{t\wedge\tau^M} \norm{ \w_\nu}^2_{L^2} \, dsdt\right) \left(\widetilde{\E}\int_{\delta\wedge \tau^M}^{T\wedge \tau^M}\!\!\int_{\delta\wedge \tau^M}^{t\wedge\tau^M} \norm{ \w_\nu}^2_{L^2} \, dsdt\right)^2
			+ \nu^{-1}T  \norm{f_\nu}^2_{L^2(0,T;\HH)}  + \bar{\rho}T^2 .
	\end{multline*}
		We multiply this inequality by $\nu$ and use that the left hand side is nonnegative, thus after rearranging 
	\begin{equation*}
	\gamma\left(\widetilde{\E}\int_{\delta\wedge \tau^M}^{T\wedge \tau^M}\!\!\int_{\delta\wedge \tau^M}^{t\wedge\tau^M} \norm{ \w_\nu}^2_{L^2} \, dsdt\right) \left(\nu\widetilde{\E}\int_{\delta\wedge \tau^M}^{T\wedge \tau^M}\!\!\int_{\delta\wedge \tau^M}^{t\wedge\tau^M} \norm{ \w_\nu}^2_{L^2} \, dsdt\right)^2 \le
		\frac{C}{\delta}  .
	\end{equation*}
	To show that
	\begin{equation*}
		\limsup_{\nu\to 0} \nu\widetilde{\E}\int_{\delta\wedge \tau^M}^{T\wedge \tau^M}\!\!\int_{\delta\wedge \tau^M}^{t\wedge\tau^M} \norm{ \w_\nu}^2_{L^2} \, dsdt =0,
	\end{equation*}
	we follow the argument in~\cite[Lemma 3.11]{Jin2024}. We denote 
	$$
	\zeta_\nu=\nu\widetilde{\E}\int_{\delta\wedge \tau^M}^{T\wedge \tau^M}\!\!\int_{\delta\wedge \tau^M}^{t\wedge\tau^M} \norm{ \w_\nu}^2_{L^2} \, dsdt,
	$$
	and assume by contradiction that $\limsup_{\nu\to 0} \zeta_\nu \geq 2 \epsilon_0>0$. Then we can find a subsequence $\{\nu_k\}_{k\in \N}$ such that $\zeta_{\nu_k}\geq \epsilon_0$ for all $k\in \N$. Thus, we must have that
	\begin{equation*}
	\epsilon_0^2\gamma_0\leq 	\epsilon_0^2 \gamma\left(\frac{\epsilon_0}{\nu_k}\right)\leq \zeta_{\nu_k}^2 \gamma\left(\frac{\zeta_{\nu_k}}{\nu_k}\right)\leq \frac{C}{\delta}.
	\end{equation*}
	As $\nu_k\to 0$, we must have ${\epsilon_0}/{\nu_k}\to \infty$, thus by the properties of $\gamma$, we obtain $\lim_{k\to \infty}\gamma(\zeta_{\nu_k}/\nu_k)=\infty$. Hence the left hand side goes to infinity as $k\to \infty$ while the right hand side is uniformly bounded by $M/\delta<\infty$.

		This is a contradiction and hence we must have $\limsup_{\nu\to 0}\zeta_\nu = 0$. This proves that VIa and hence the dissipation term VI  vanishes. 
Thus passing to the limit $\nu\to 0$ in all the terms in~\eqref{eq:doweneedalabel}, we obtain
\begin{multline}\label{eq:Mapprox}
	 \underbrace{\widetilde{\E}\int_0^{T\wedge \tau^M}\theta(t)\norm{\widetilde{u}(t)}_{L^2(\dom)}^2 dt}_{\text{(i)}}
= \underbrace{\widetilde{\E}\int_0^{T\wedge \tau^M}\theta(t)\norm{\widetilde{\bar{u}}}_{L^2(\dom)}^2 dt }_{\text{(ii)}}
	\\
	 +  2\underbrace{\widetilde{\E}\int_0^{T\wedge \tau^M}\theta(t)\int_{0}^{t \wedge \tau^M}(f,\widetilde{u}) ds dt }_{\text{(iii)}} + \underbrace{\widetilde{\E}\int_0^{T\wedge \tau^M}\theta(t)\bar{\sigma} t dt }_{\text{(iv)}} 
	\\
	  +   2\underbrace{\widetilde{\E}\int_{0}^{T\wedge \tau^M}\theta(t)\int_{0}^{t\wedge \tau^M} \widetilde{u} \sigma d\widetilde{W}(s) 		 dt  }_{\text{(v)}}+\mathcal{O}(\epsilon).
\end{multline}
Since $\epsilon>0$ was arbitrary, this holds as an equality.
 Next, we want to send $M\to \infty$. We  will show first that
\begin{equation}\label{eq:neverstops}
\widetilde{\P}[\lim_{M\to \infty} \tau^M = T] = 1,
\end{equation}
thus $\tau^M\to T$ as $M\to \infty$ almost surely. 
We note that by Chebyshev inequality, for any $M>0$, we have
\begin{equation*}
	\widetilde{\P}[\tau^M < T ]  = \widetilde{\P}\left[\int_0^T \norm{\widetilde{u}(s)}_{L^2}^2 ds \geq M\right] \leq \frac{1}{M}\widetilde{\E}\int_0^T \norm{\widetilde{u}(s)}_{L^2}^2 ds \leq \frac{C}{M},
\end{equation*}
where the last inequality follows (for example) from~\eqref{eq:linfl2ubound}. Sending $M\to \infty$, we obtain in this inequality
\begin{equation*}
	\lim_{M\to\infty} 	\widetilde{\P}[\tau^M < T ] = 0.
\end{equation*}
Since $\tau^M$ is monotone increasing in $M$, we conclude that
\begin{equation*}
	\widetilde{\P}[\lim_{M\to \infty} \tau^M < T] = 	\lim_{M\to\infty} 	\widetilde{\P}[\tau^M < T ] = 0,
\end{equation*}
thus~\eqref{eq:neverstops} holds. Now we can send $M\to\infty$ in all the terms in~\eqref{eq:Mapprox}. We assume for the moment that $\theta$ is nonnegative.  For the first term, (i), since $\int_0^{T\wedge \tau^M}\theta(t)\norm{\widetilde{u}(t)}_{L^2}^2 dt$ is increasing in $M$ and converges to $\int_0^{T}\theta(t)\norm{\widetilde{u}(t)}_{L^2}^2 dt$ almost surely, we can use the monotone convergence theorem and  pass to the limit $M\to \infty$ and obtain
\begin{equation*}
	\text{(i)}\stackrel{M\to \infty}{\longrightarrow}\widetilde{\E}\int_0^{T}\theta(t)\norm{\widetilde{u}(t)}_{L^2}^2 dt.
\end{equation*}
In the same way, we establish convergence of terms (ii) and (iv):
\begin{equation*}
	\text{(ii)}\stackrel{M\to \infty}{\longrightarrow}\widetilde{\E}\int_0^{T}\theta(t)\norm{\widetilde{\bar{u}}}_{L^2}^2 dt, \quad 	\text{(iv)}\stackrel{M\to \infty}{\longrightarrow}\widetilde{\E}\int_0^{T}\theta(t)\bar{\sigma} t  dt.
\end{equation*}
	For terms (iii) and (v), we first note that we can replace $t\wedge \tau^M$ in the inner integral by $t$. 
	Then, we bound the integrands uniformly in $\tau^M$. For (iii), we have
	\begin{equation*}
		\left|\int_0^{T\wedge \tau^M}\theta(t) \int_0^t (f,\widetilde{u}) ds dt \right| \leq 
		T\norm{\theta}_{L^\infty}\norm{f}_{L^2(0,T;\HH)} \norm{\widetilde{u}}_{L^2(0,T;\HH)},
	\end{equation*}
	which is integrable with respect to $d\widetilde{\P}$ by~\eqref{eq:linfl2ubound}, thus we can use the dominated convergence theorem to pass to the limit in (iii). For the integrand in (v), we have by the Cauchy-Schwarz inequality
	\begin{equation*}\begin{split}
	 \left|	\int_{0}^{T\wedge \tau^M}\theta(t)\int_{0}^{t} \widetilde{u} \sigma d\widetilde{W}(s) 		 dt\right| &\leq \left(\int_0^{T\wedge \tau^M}|\theta(t)|^2 dt\right)^{1/2}\left(\int_0^{T\wedge\tau^M}\left(\int_0^t\widetilde{u}\sigma d\widetilde{W}(s)\right)^2 dt \right)^{1/2}\\
	&\leq \sqrt{T}\norm{\theta}_{L^\infty}\left(\int_0^{T}\left(\int_0^t\widetilde{u}\sigma d\widetilde{W}(s)\right)^2 dt \right)^{1/2}
		\end{split}
	\end{equation*}
	which is integrable in $d\widetilde{\P}$ thanks to the It\^o isometry~\cite[Equation (4.30)]{DaPrato2014} and the uniform bound~\eqref{eq:linfl2ubound}. Thus, we can also send $M\to\infty$ in the term (v) using the dominated convergence theorem once more. Combining all of this, we obtain in the limit $M\to\infty$ for~\eqref{eq:Mapprox},
	\begin{multline}\label{eq:nomoreapprox}
		\widetilde{\E}\int_0^{T}\theta(t)\norm{\widetilde{u}(t)}_{L^2(\dom)}^2 dt
		= \widetilde{\E}\int_0^{T}\theta(t)\norm{\widetilde{\bar{u}}}_{L^2(\dom)}^2 dt 
		\\
		+  2\widetilde{\E}\int_0^{T}\theta(t)\int_{0}^{t }(f,\widetilde{u}) ds dt+ 2\widetilde{\E}\int_0^{T}\theta(t)\bar{\sigma} t dt 
		+   2\widetilde{\E}\int_{0}^{T}\theta(t)\int_{0}^{t} \widetilde{u} \sigma d\widetilde{W}(s) 		 dt .
	\end{multline}
	For nonpositive $\theta$, we can multiply identity~\eqref{eq:Mapprox} by $-1$ and repeat the argument above and general $\theta\in C([0,T])$ can be decomposed into their  nonnegative and nonpositive part via $\theta = \theta^+-\theta^-$ where $\theta^+(t) = \max(0,\theta(t))$ and $\theta^-(t)=\min(0,\theta(t))$. Note that the last term in~\eqref{eq:nomoreapprox} is zero since $\widetilde{\E}\int_0^t \widetilde{u}\sigma d\widetilde{W}=0$.
Now finally $\theta$ gets its big moment. 
Since $\theta\in C([0,T])$ can be chosen arbitrary, it follows from~\eqref{eq:nomoreapprox} that for almost all $t\in [0,T]$ it holds that
	\begin{equation*}
			\widetilde{\E}\norm{\widetilde{u}(t)}_{L^2(\dom)}^2 
		= \widetilde{\E}\norm{\widetilde{\bar{u}}}_{L^2(\dom)}^2 
		+  2\widetilde{\E}\int_{0}^{t }(f,\widetilde{u}) ds + \bar{\sigma} t 	,
	\end{equation*}
	which is what we wanted to prove.

\end{proof}
\begin{remark}[Algebraic decay of structure functions]
	If the modulus of continuity is of the form $\phi(r) = r^\alpha$ for some $\alpha>0$, one can replace condition~\eqref{eq:structurefcn} by
	\begin{equation}\label{eq:algebraicdecay}
		\E S^T_2(u_\nu,r)^2\leq C \phi(r)^2,\quad \forall \, 0\leq r\leq r_0.
	\end{equation}
	In fact, condition~\eqref{eq:structurefcn} is used in its precise form only to derive tightness of the measures $\mu_\nu = \P\circ u_\nu^{-1}$  on $L^2(0,T;\HH)$ as in~\eqref{eq:tightness}. Condition~\eqref{eq:algebraicdecay} implies that $\E \int_0^T |u_\nu|_{H^s}^2 dt<\infty$ uniformly in $\nu>0$ for $0<s<\alpha$ as we will show in the upcoming Lemma~\ref{Sobolev-str_fcn}. Since $H^s(\dom)$ is a compact subset of $L^2(\dom)$ for $s>0$, one can derive 
		\begin{equation*}
		\P\left[   	\sup_{t\in[0,T]}\norm{u_\nu}_{L^2}^p+ \int_0^T |u_\nu|_{H^s}^2 dt +	\norm{u_\nu}_{Z}\leq M_\epsilon\right]\geq 1-\epsilon.
	\end{equation*}
 for any $\epsilon>0$	instead of~\eqref{eq:tightness}, which then implies tightness of the measures $\mu_\nu$ on $L^2(0,T;\HH)$ using a variant of Proposition~\ref{prop:moctimedep} (c.f. also~\cite[Theorem 2.1]{Flandoli1995}).
\end{remark}

\begin{lemma} \label{Sobolev-str_fcn}
	The following relation between structure function and Sobolev spaces holds:
	$$\E\int_0^T |v(t)|_{W^{\alpha,p}}^p dt  <\infty \implies \forall \; r>0: \E S_p^T(v(t);r)^p\le Cr^{p\alpha} \implies \forall \; s<\alpha: \E\int_0^T|v(t)|_{W^{s,p}}^p dt <\infty.$$
\end{lemma}

\begin{proof}
	For the first implication observe that
	\begin{align*}
		\left( \frac{S_p(v;r)}{r^\alpha} \right)^p 
		&= C \int_{D}\int_{B_r(0)} \frac{|v(x+h)-v(x)|^p}{r^{d+p\alpha}} \, dh \, dx\\
		&\le C \int_{D}\int_{D} \frac{|v(x+h)-v(x)|^p}{|h|^{d+p\alpha}} \, dh \, dx\\
		&= C |v|_{W^{\alpha,p}}^p,
	\end{align*}
	where we used the Sobolev--Slobodeckij seminorm in the last line. Integrating over time and taking expectations yields the first implication.
	
	For the second implication let $r_0 = \diam(D) = \frac{\sqrt{2}}{2}$ and $r_i = 2^{-i}r_0$ for $i>0$, and denote by $A_{r_i}$ the annulus around $0$ with radii $r_{i}$ and $r_{i+1}$. Then
	\begin{align*}
		\E\int_0^T|v|_{W^{s,p}}^pdt
		&= \E\int_0^T\int_{D}\int_{D} \frac{|v(x+h)-v(x)|^p}{|h|^{d+sp}} \, dh \, dx\, dt\\
		&\le \E\int_0^T\int_{D}\sum_{i=0}^\infty \int_{A_{r_{i}}} \frac{|v(x+h)-v(x)|^p}{r_{i+1}^{d+sp}} \, dh \, dx\, dt\\
		&\le C\sum_{i=0}^\infty r_{i+1}^{-sp} \E \int_0^T \int_{D}\fint_{B_{r_i}} |v(x+h)-v(x)|^p dt \, dh \, dx\\
		&= C\sum_{i=0}^\infty r_{i+1}^{-sp}\, \E \,S_p^T(v;r_i)^p\\
		&\le C\sum_{i=0}^\infty r_{i+1}^{(\alpha - s)p}\\
		&< \infty,
	\end{align*}
    for any $s<\alpha$.

\end{proof}

\begin{proposition}
	\label{prop:energybalancetotightness}
	Let $\{u_\nu\}_{\nu>0}$ be a sequence of solutions of the stochastic Navier-Stokes equations as in Definition~\ref{def:strongsolNS}. 
	Assume that $f_\nu\to f$ in $L^1(0,T;\HH)$ with $f,f_\nu\in L^2(0,T;\HH)$ and that $\bar{u}_\nu\to \bar{u}$ in $L^2(\Omega;L^2_{\Div}(\dom))$.
	Assume that the laws $\mu_\nu: =\P\circ u_\nu^{-1}$ are weakly tight on $\mathcal{P}(C([0,T];\HH_w)\cap L^2(0,T;\HH)_w)$ and converge weakly in $\mathcal{P}(C([0,T];\HH_w)\cap L^2(0,T;\HH)_w)$ to $\mu:=\P\circ u^{-1}$ that satisfies
		\begin{equation}
		\label{eq:energybalanceeulerprop}
		\E \norm{u(t)}_{L^2}^2 = \E\norm{\bar{u}}_{L^2}^2 +2\E\int_0^t (u,f) ds +\bar{\sigma} t.
	\end{equation}
	Then $\mu_\nu$ are a tight sequence on $\mathcal{P}(L^2(0,T;\HH))$, i.e., with respect to the strong topology on $L^2(0,T;\HH)$. 
\end{proposition}

\begin{proof}
	We pick a subsequence, still denoted by $\nu$ such that $\mu_\nu\to\mu$ weakly on $\mathcal{P}(C(0,T;\HH_w)\cap L^2(0,T;\HH)_w)$ and $f_\nu\to f$ in $L^1(0,T;\HH)$. Denote $X: =C(0,T;\HH_w)\cap L^2(0,T;\HH)$. 
We subtract~\eqref{eq:energybalanceeulerprop} from~\eqref{eq:NSenergybalance} to obtain
\begin{equation*}
		\begin{split}
		\E \norm{u_\nu(t)}_{L^2}^2-\E \norm{u(t)}_{L^2}^2 &= \E \norm{\bar{u}_\nu}_{L^2}^2 - 2\nu\E\int_0^t \norm{\Grad u_\nu(s)}_{L^2}^2 ds + 2\E\int_0^t (u_\nu , f_{\nu} ) ds -\E\norm{\bar{u}}_{L^2}^2 -2\E\int_0^t (u,f) ds\\
		& \leq  \E \norm{\bar{u}_\nu}_{L^2}^2  + 2\E\int_0^t (u_\nu , f_{\nu} ) ds -\E\norm{\bar{u}}_{L^2}^2 -2\E\int_0^t (u,f) ds.
		\end{split}
\end{equation*}
Using the assumption on the convergence of the initial data, we have 
\begin{equation*}
	\E\norm{\bar{u}_\nu}_{L^2}^2 - \E\norm{\bar{u}}_{L^2}^2 \stackrel{\nu\to 0}{\longrightarrow } 0.
\end{equation*}
Since $f_\nu\to f$ in $L^1(0,T;L^2(\dom))$, we have
\begin{equation*}
\left|	\E\int_0^t (u_\nu,f-f_\nu) ds\right| \leq \E\sup_{t\in [0,T]} \norm{u_\nu}_{L^2(\dom)}\int_0^t \norm{f-f_\nu}_{L^2(\dom)} ds \to 0,\quad \text{as }\,\nu\to 0,
\end{equation*}
using~\eqref{eq:linfl2ubound}.
Moreover,
\begin{equation*}
	\E \int_0^t (u-u_\nu,f) ds=\int_X \int_0^t (u,f) ds d\mu(u) -  \int_X \int_0^t (u,f) ds d\mu_\nu(u) \to 0,\quad \text{as }\,\nu\to 0,
\end{equation*}
since $\mu_\nu$ converges to $\mu$ in $\mathcal{P}(X_{w})$. This follows from~\cite[Lemma 5.1.12 (d)]{Ambrosio2008}, since
\begin{equation*}
	\psi(u) = \int_0^t(f,u)ds
\end{equation*}
is weakly continuous on bounded sets of $X$ and $|\psi(u)|$ is uniformly integrable on $X$. 
Thus
\begin{equation*}
	\E\int_0^t(f_\nu,u_\nu) ds-\E\int_0^t (f,u) ds \to 0,\quad \text{as }\, \nu\to 0.
\end{equation*}
We obtain that
\begin{equation}\label{eq:normconv}
	\E \norm{u_\nu(t)}_{L^2}^2 = \int_X \norm{u(t)}_{L^2}^2d\mu_\nu(u)\stackrel{\nu\to 0}{\longrightarrow}\int_X\norm{u(t)}_{L^2}^2 d\mu(u) = \E\norm{u(t)}_{L^2}^2.	
\end{equation}
Ths is true for any $t\in [0,T]$. Since $\E \norm{u_\nu(t)}_{L^2}^2$ are uniformly bounded by the energy balance, we can use the dominated convergence theorem to derive
\begin{equation*}
	\int_0^T\E \norm{u_\nu(t)}_{L^2}^2 dt =  \int_X \norm{u}_{L_t^2L_x^2}^2 d\mu_\nu(u)\stackrel{\nu\to 0}{\longrightarrow}\int_X\norm{u}_{L_t^2L^2_x}^2 d\mu(u) = \int_0^T\E\norm{u(t)}_{L^2}^2dt.
\end{equation*}
Now we can apply~\cite[Theorem 5.1.13]{Ambrosio2008} with $X=L^2(0,T;\HH)$ and $j(x) = x^2$ to obtain that $\mu_\nu\to \mu$ on $\mathcal{P}(X)$, in particular, with respect to the strong topology on $L^2(0,T;\HH)$.

\end{proof}
\begin{remark}
	We note that in this case $t\mapsto \E\norm{u(t)}_{L^2}^2$ is continuous thanks to the assumption of the energy balance~\eqref{eq:energybalanceeulerprop} and the fact that the right hand side of~\eqref{eq:energybalanceeulerprop} is continuous with respect to $t$. 
	Moreover, we have that $t\mapsto \E (u(t),v)$ is continuous for any $v\in L^2(\Omega;\HH)$ that is measurable with respect to $\mathcal{F}$ since $u$ is weakly continuous with values in $\HH$, c.f. the argument for~\eqref{eq:weakcont}. Thus
	\begin{equation*}
	\lim_{s\to t}	\E \norm{u(t)-u(s)}_{L^2}^2 = 	\E\norm{u(t)}_{L^2}^2 + \lim_{s\to t}	\left(\E \norm{u(s)}_{L^2}^2 -2 \E (u(t),u(s))\right) = 0,
	\end{equation*}
	which implies that $\norm{u(t)-u(s)}_{L^2}^2\to 0$ a.s. as $s\to t$. Thus $u\in C(0,T;\HH)$ almost surely. 
	\end{remark}

\section{Numerical Experiments}\label{sec:num}

We simulate the Navier-Stokes equations \eqref{SNS} in the domain $[0,1]^2$ with periodic boundary conditions using the highly efficient GPU based solver azeban \cite{rohner2024}.

Fix a mesh width $\delta = \frac{1}{N}$ for some $N \in \mathbb{N}$. We consider the following spectral discretization of the Navier-Stokes equations in the Fourier domain
\begin{equation} \label{eq:shv_disc}
    \begin{cases}
        du_{\nu}^{\delta} + \mathcal{P}_N(u_{\nu}^{\delta}\cdot\nabla u_{\nu}^{\delta}) + \nabla p_{\nu}^{\delta} &= \nu\Delta u_{\nu}^{\delta}\,dt + f\,dt + \sigma\cdot dW  \\
        \Div u_{\nu}^{\delta} &= 0 \\
        u_{\nu}^{\delta}|_{t=0} &= \mathcal{P}_N\bar{u}
    \end{cases}
\end{equation}
where $\mathcal{P}_N$ is the spatial Fourier projection operator mapping a function $g(x,t)$ to its first $N$ Fourier modes: $\mathcal{P}_N = \sum_{|k|_{\infty}\leq N} \hat{g}_k(t)e^{ik\cdot x}$.

Note that this scheme merely conserves the divergence of $u_{\nu}^{\delta}$ instead of setting it to zero. For this reason, we implicitly project all initial conditions onto divergence-free vector fields using the Leray projection.

For the sake of the experiments, we will choose the basis for the stochastic forcing term as

\begin{equation}
    e_{i,j}(x_1,x_2) = \frac{2}{\sqrt{i^2+j^2}} \begin{bmatrix}
        j \cos(2\pi ix_1) \cos(2\pi jx_2) \\
        i \sin(2\pi ix_1) \sin(2\pi jx_2)
    \end{bmatrix}
\end{equation}
indexed by $i,j \in \mathbb{N}$ and set all $b_{i,j}$ corresponding to a basis function with one of the indices $i,j > N_b \in \mathbb{N}$ to zero. All other $b_{i,j}$ are chosen to be equal to $\sigma \in \mathbb{R}_+$. We further set the deterministic forcing term $f=0$. We therefore expect the energy of the system to evolve according to 
\begin{equation}
    \mathbb{E}\|u(t)\|_{L^2}^2 + 2\nu \mathbb{E} \int_0^t \|\nabla u(s)\|_{L^2}^2 \, ds = \mathbb{E}\|\bar{u}\|^2_{L^2} + \bar{\sigma}t, \quad t\ge 0,
\end{equation}
where $\bar{\sigma} = \sum_{i,j} b_{i,j}^2 = N_b^2\sigma^2$ and $\nu$ the artificial viscosity of our numerical solver. For the experiments performed here, we will set $N_b = 9$ leaving us with a total of $9^2 = 81$ non-zero basis functions each having a coefficient $\sigma$. The energy is therefore expected to rise linearly in time with slope $\bar{\sigma} = 81\sigma^2$. To experimentally verify this theoretical value of $\bar{\sigma}$, we compute solutions to the initial conditions given below and compute the value of
\begin{equation}
    \mathbb{E}\|u(t)\|_{L^2}^2 + 2\nu \mathbb{E} \int_0^t \|\nabla u(s)\|_{L^2}^2 \, ds - \mathbb{E}\|\bar{u}\|^2_{L^2}.
\end{equation}
The integral over the gradient is hereby approximated by a simple Riemann sum over a total of 10000 rectangles (40000 for the flat vortex sheet experiment with $\sigma = 0.1$). The norm of the gradient itself is computed in Fourier space making use of Parseval's identity.

Furthermore, all simulations are performed with the three different viscosities $\nu = \frac{0.05}{N}$, $\nu = \frac{0.1}{N}$, and $\nu = \frac{0.2}{N}$ where $N$ is the resolution of the computational grid ($N = 256$ unless otherwise noted). We use 32 different realizations of each experiment.

\subsection{Flat Vortex Sheet}

As a first experiment, we use the Flat Vortex Sheet initial condition from \cite{LMP2021}. The initial data is given by
\begin{equation}
    \begin{aligned}
        \bar{u}^{(1)}(x_1, x_2) &= \begin{cases}
            \tanh\left(2\pi\frac{x_2-0.25}{\rho}\right) &\mbox{ for } x_2+\sigma_{\delta}(x_1) \leq \frac{1}{2} \\
            \tanh\left(2\pi\frac{0.75-x_2}{\rho}\right) &\mbox{ otherwise}
        \end{cases} \\
        \bar{u}^{(2)}(x_1, x_2) &= 0
    \end{aligned}
\end{equation}
where $\sigma_{\delta}: [0,1] \to \mathbb{R}$ is a perturbation of the initial data given by
\begin{equation}
    \sigma_{\delta}(x_1) = \delta \sum_{k=1}^{p} \alpha_k\sin(2\pi kx_1 - \beta_k).
\end{equation}
The random variables $\alpha_k$ and $\beta_k$ are i.i.d. uniformly distributed on $[0,1]$ and $[0,2\pi]$ respectively. The parameters $\delta$ and $p$ are chosen to be $\delta = 0.025$ and $p = 10$. The initial condition is also well defined in the limit $\rho \to 0$ in which case, however, it becomes discontinuous. Here, we perform the experiments with $\rho = 0.1$.

\begin{figure}[ht!]
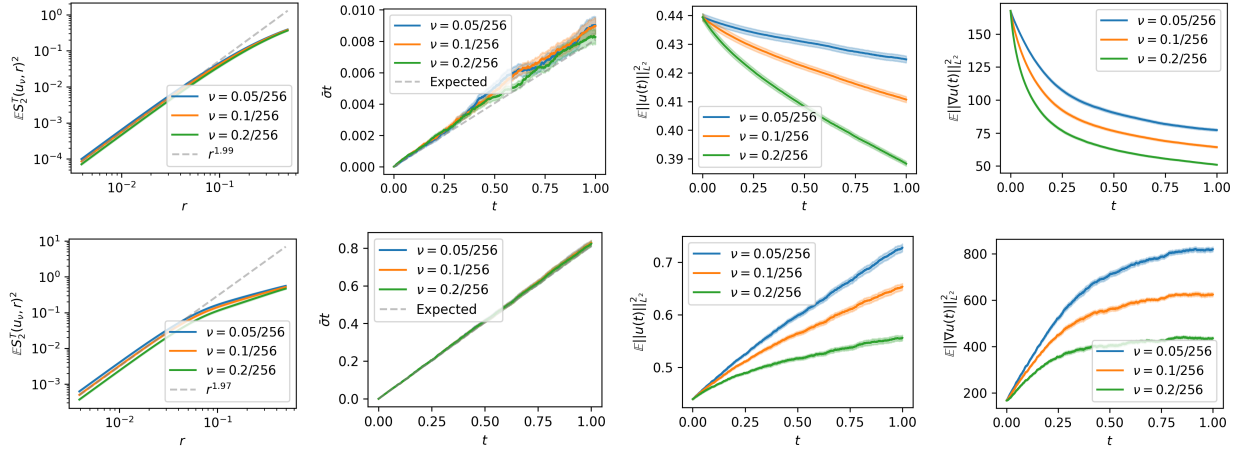

    \centering
    \foreach \s in {0.01, 0.1} {
        \begin{subfigure}{0.245\textwidth}
            \includegraphics[width=\textwidth]{img/ddsl/sigma\s/S2T.png}
        \end{subfigure}
        \begin{subfigure}{0.245\textwidth}
            \includegraphics[width=\textwidth]{img/ddsl/sigma\s/Bt.png}
        \end{subfigure}
        \begin{subfigure}{0.245\textwidth}
            \includegraphics[width=\textwidth]{img/ddsl/sigma\s/E.png}
        \end{subfigure}
        \begin{subfigure}{0.245\textwidth}
            \includegraphics[width=\textwidth]{img/ddsl/sigma\s/grad_L2.png}
        \end{subfigure}
    }
    \caption{Behavior of key quantities of interest for the flat vortex sheet experiment. The top row shows the results for $\sigma = 0.01$, while the bottom row does so for $\sigma = 0.1$. The plots depict from left to right: Second-order time-integrated structure function, measured and predicted energy input due to forcing, total kinetic energy, energy dissipation due to viscosity. The shaded region in the plots depicts one standard deviation in the Monte Carlo approximation of the mean. One can see that the structure functions clearly have a uniform modulus of continuity. Therefore, the experiments are expected to satisfy the energy balance equation. This is verified by the plots in the second column, demonstrating that the actual energy input perfectly follows the predicted one.}
\end{figure}

\subsection{Fractional Brownian Bridges}

This experiment is taken from \cite{Lanthaler2021}. We take Fractional Brownian Bridges as initial measures for each of the velocity components of $u$. These objects can be seen as generalizations for Brownian Bridges which have a tunable energy spectrum $E_k \sim k^{-(2H+1)}$ (in 2D) where the parameter $0 < H < 1$ is called the \textit{Hurst index}. Choosing a higher Hurst index has a regularizing effect on the Brownian Bridge, while a lower Hurst index increases the roughness of the initial condition. For $H = 0.5$, we obtain a standard Brownian Bridge. A realization of a Fractional Brownian Bridge is a fractal object and thus not covered by assumptions of any well-posedness results of the incompressible Euler equations, not even in the case of the 2D equations. Nevertheless, we can numerically verify whether the approximations satisfy a condition like~\eqref{eq:structurefcn} and experimentally test whether the energy balance~\eqref{eq:energybalanceintro} holds.

We generate Brownian bridges directly in Fourier space with the following method:
\begin{equation}
    W(x_1, x_2) = \sum_{|\mathbf{k}|_{\infty} \leq N} \frac{1}{\left\|\mathbf{k}\right\|_2^{H+1}} \sum_{m,n\in\{0,1\}} \alpha_k^{(mn)}\text{sc}_{m}(x_1)\text{sc}_{n}(x_2)
\end{equation}
where
\begin{equation}
    \text{sc}_i(x) = \begin{cases}
        \sin(x) &\mbox{ for } i = 0 \\
        \cos(x) &\mbox{ for } i = 1
    \end{cases}
\end{equation}
and the $\alpha_k^{(mn)} \sim \mathcal{U}_{[-1,1]}$.

The forcing term is kept the same as before. In the following figures, first columns, we observe numerically that the solutions have a modulus of continuity in mean. We thus expect the evolution of the energy to be the same as before, independent of the initial condition. We observe almost perfect correspondence between the theoretical evolution of the energy and the energy of the simulated system. It is probable that part of the deviation stems from the discretization of the integral computing the loss of energy through dissipation due to viscosity. Increasing the number of samples generally improved the accuracy of the fit (especially for the shear layer), but simultaneously increased memory requirements. Discretizing using 10000 samples was found to be a good compromise. We note that as the viscosity $\nu$ is decreased, $\E \norm{u(t)}_{L^2}^2$ more closely follows $\bar{\sigma} t$.
 
\foreach \H in {15, 50, 75} {
    \begin{figure}[ht!]
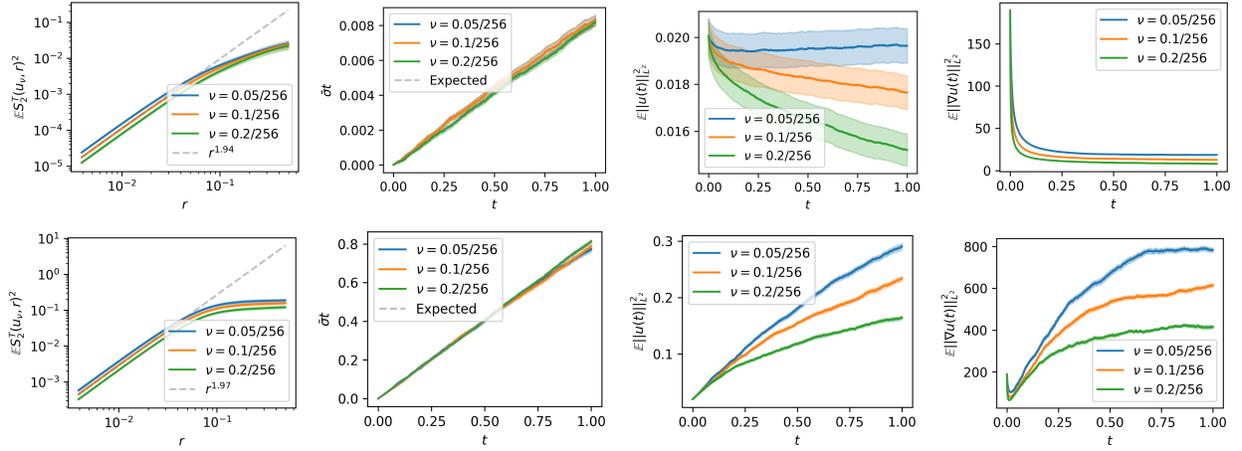

        \centering
        \foreach \s in {0.01, 0.1} {
            \begin{subfigure}{0.245\textwidth}
                \includegraphics[width=\textwidth]{img/bm/H\H/sigma\s/S2T.png}
            \end{subfigure}
            \begin{subfigure}{0.245\textwidth}
                \includegraphics[width=\textwidth]{img/bm/H\H/sigma\s/Bt.png}
            \end{subfigure}
            \begin{subfigure}{0.245\textwidth}
                \includegraphics[width=\textwidth]{img/bm/H\H/sigma\s/E.png}
            \end{subfigure}
            \begin{subfigure}{0.245\textwidth}
                \includegraphics[width=\textwidth]{img/bm/H\H/sigma\s/grad_L2.png}
            \end{subfigure}
        }
        \caption{Results of the fractional Brownian bridge experiment with Hurst index $H = 0.\H$. The plots in the top row correspond to forcing strength $\sigma = 0.01$, and the bottom row to $\sigma = 0.1$. Furthermore, the depicted quantities from left to right are: Second-order time-integrated structure functions, measured and predicted energy input due to forcing, total kinetic energy, energy dissipation due to viscosity. The shaded region depicts one standard deviation around the Monte Carlo approximation of the mean.}
    \end{figure}
}
\newpage\null\newpage
\appendix

\section{Auxiliary results}
\begin{lemma}
	\label{lem:tsupL2boundu}
	Under the assumption that $\E\norm{\bar{u}_\nu}_{L^2}^p<\infty$ for some $p\geq 2$,
	solutions $u_\nu$ of the Navier-Stokes equations~\eqref{SNS} satisfy uniformly in $\nu$
	\begin{equation*}
	\E\sup_{t\in [0,T]}\norm{u_\nu}_{L^2(\dom)}^p\leq C(p).
	\end{equation*}
\end{lemma}
\begin{proof}
	We consider a Galerkin approximation $u^m_\nu$ of $u_\nu$ in order to be able to do any formal manipulations (c.f. the proof of Lemma~\ref{lem:vorticitybound}, equation~\eqref{eq:galerkinprojection} and ff.)).
	By It\^o's formula, we have for $p\geq 2$, 
	\begin{equation*}
	d\norm{u^m_\nu}_{L^2}^p\leq p\norm{u_\nu^m}^{p-2}_{L^2}(u_\nu^m,d u_\nu^m)+\frac{1}{2}p(p-1)\norm{u_\nu^m}^{p-2}_{L^2}\norm{\sigma^m}_{L^2}^2dt.
	\end{equation*}
	Thus,
	\begin{multline*}
	d\norm{u^m_\nu}_{L^2}^p +\nu p\norm{u_\nu^m}^{p-2}_{L^2}\norm{\Grad u^m_\nu}_{L^2}^2 dt\leq p\norm{u_\nu^m}^{p-2}_{L^2}(f^m_\nu,u^m_\nu)dt +p\norm{u_\nu^m}^{p-2}_{L^2}(u^m_\nu,\sigma^m)dW\\
	+\frac{1}{2}p(p-1)\norm{u_\nu^m}^{p-2}_{L^2}\norm{\sigma^m}_{L^2}^2dt.
	\end{multline*}
	since $b(u_\nu^m,u_\nu^m,u_\nu^m)=0$. Using Young's inequality, we obtain

	\begin{multline*}
	d\norm{u^m_\nu}_{L^2}^p +\nu p\norm{u_\nu^m}^{p-2}_{L^2}\norm{\Grad u^m_\nu}_{L^2}^2 dt\leq \frac{p}{2}\norm{u_\nu^m}^{p}_{L^2}dt +\frac{p-2}{2}\norm{u_\nu^m}_{L^2}^p\norm{f^m_\nu}_{L^2}^2 dt+\norm{f^m_\nu}_{L^2}^2 dt\\
	+p\norm{u_\nu^m}^{p-2}_{L^2}(u^m_\nu,\sigma^m)dW
	+\left(\frac{1}{2}(p-2)(p-1)\norm{u_\nu^m}^{p}_{L^2}+ p-1\right)\norm{\sigma^m}_{L^2}^2dt.
\end{multline*}
	We integrate in time and take the supremum over time:

	\begin{multline*}
	\sup_{0\leq s\leq t}\norm{u^m_\nu(s)}_{L^2}^p +\nu p\int_0^t\norm{u_\nu^m}^{p-2}_{L^2}\norm{\Grad u^m_\nu}_{L^2}^2 d\tau
	\\
	\leq \norm{u^m_\nu(0)}_{L^2}^p+ \frac{p}{2}\sup_{0\leq s\leq t}\int_0^s\norm{u_\nu^m}^{p}_{L^2}d\tau +\frac{p-2}{2}\sup_{0\leq s\leq t}\int_0^s\norm{f^m_\nu}_{L^2}^2 \norm{u_\nu^m}_{L^2}^p d\tau
	+\int_0^t\norm{f^m_\nu}_{L^2}^2d\tau\\
	+p\sup_{0\leq s\leq t}\int_0^s \norm{u_\nu^m}^{p-2}_{L^2}(u^m_\nu,\sigma^m)dW(\tau)
	+\sup_{0\leq s\leq t}\int_0^s\left(\frac{1}{2}(p-2)(p-1)\norm{u_\nu^m}^{p}_{L^2}+ p-1\right)\norm{\sigma^m}_{L^2}^2d\tau.
\end{multline*}
	We take expectations and then use the Burkholder-Davis-Gundy inequality for the term involving the Brownian motion:
	\begin{equation*}
	\begin{split}
	p\E \sup_{0\leq s\leq t}\left|\int_0^s \norm{u_\nu^m}^{p-2}_{L^2}(u^m_\nu,\sigma^m)dW(\tau)\right|&\leq Cp \E\left(\int_0^t\norm{u^m_\nu}_{L^2}^{2p-2}\norm{\sigma^m}_{L^2}^2 d\tau\right)^{\frac12}\\
	&\leq \frac12\E\sup_{0\leq s\leq t}\norm{u^m_\nu}_{L^2}^p+ Cp^2\int_0^t\E\norm{u_\nu^m}_{L^2}^{p-2}\norm{\sigma^m}_{L^2}^2 d\tau\\
	&\leq \frac12\E\sup_{0\leq s\leq t}\norm{u^m_\nu}_{L^2}^p+ Cp^2\int_0^t\left(\E\norm{u_\nu^m}_{L^2}^{p}+1\right)\norm{\sigma^m}_{L^2}^2 d\tau\\
	&\leq \frac12\E\sup_{0\leq s\leq t}\norm{u^m_\nu}_{L^2}^p+ Cp^2\int_0^t \E\norm{u_\nu^m}_{L^2}^{p} \norm{\sigma^m}_{L^2}^2 d\tau +Cp\int_0^t \norm{\sigma^m}_{L^2}^2 d\tau.
	\end{split}
	\end{equation*}
	Thus, we obtain

	\begin{multline*}
	\frac12\E\sup_{0\leq s\leq t}\norm{u^m_\nu(s)}_{L^2}^p +\nu p\E\int_0^t\norm{u_\nu^m}^{p-2}_{L^2}\norm{\Grad u^m_\nu}_{L^2}^2 d\tau
	\\
	\leq \E\norm{u^m_\nu(0)}_{L^2}^p+ \frac{p}{2}\int_0^t\E\sup_{0\leq s\leq \tau}\norm{u_\nu^m(s)}^{p}_{L^2}d\tau
	+\int_0^t\norm{f^m_\nu}_{L^2}^2d\tau     +\frac{p-2}{2}\int_0^t\norm{f^m_\nu}_{L^2}^2\E \sup_{0\leq s\leq \tau} \norm{u^m_\nu(s)}_{L^2}^pd\tau\\
	+Cp^2\int_0^t \E\sup_{0\leq s\leq \tau}\norm{u_\nu^m(s)}_{L^2}^{p} \norm{\sigma^m(\tau)}_{L^2}^2 d\tau +Cp^2\int_0^t \norm{\sigma^m}_{L^2}^2 d\tau\\
	+\int_0^t\left(\frac{1}{2}(p-2)(p-1)\E\sup_{0\leq s\leq \tau}\norm{u_\nu^m}^{p}_{L^2}+ p-1\right)\norm{\sigma^m}_{L^2}^2d\tau.
\end{multline*}
	Then we obtain by using Gr\"onwall's inequality
 \begin{equation*}
     \E\sup_{0\leq s \leq t}\norm{u^m_{\nu}}_{L^2}^p\leq C.
 \end{equation*}
 The estimate is uniform in $m$ and thus we can send $m\to \infty$ and obtain that it holds in the limit also. As $u_\nu$ is unique, we have $u_\nu^m\to u_\nu$ almost surely. We further note that this estimate is uniform in $\nu$ also.
\end{proof}

\bibliographystyle{abbrv}
\bibliography{energyconservationbib}

\begin{thebibliography}{10}

\bibitem{Ambrosio2008}
L.~Ambrosio, N.~Gigli, and G.~Savar\'e.
\newblock {\em Gradient flows in metric spaces and in the space of probability measures}.
\newblock Lectures in Mathematics ETH Z\"urich. Birkh\"auser Verlag, Basel, second edition, 2008.

\bibitem{Bensoussan1973}
A.~Bensoussan and R.~Temam.
\newblock \'equations stochastiques du type {N}avier-{S}tokes.
\newblock {\em J. Functional Analysis}, 13:195--222, 1973.

\bibitem{Bessaih1999b}
H.~Bessaih.
\newblock Martingale solutions for stochastic {E}uler equations.
\newblock {\em Stochastic Anal. Appl.}, 17(5):713--725, 1999.

\bibitem{Bessaih2013}
H.~Bessaih.
\newblock Stochastic incompressible {E}uler equations in a two-dimensional domain.
\newblock In {\em Stochastic analysis: a series of lectures}, volume~68 of {\em Progr. Probab.}, pages 135--155. Birkh\"{a}user/Springer, Basel, 2015.

\bibitem{Bessaih1999}
H.~Bessaih and F.~Flandoli.
\newblock {2-D Euler equation perturbed by noise}.
\newblock {\em Nonlinear Differential Equations and Applications NoDEA}, 6(1):35--54, 1999.

\bibitem{Boyer2013}
F.~Boyer and P.~Fabrie.
\newblock {\em Mathematical tools for the study of the incompressible {N}avier-{S}tokes equations and related models}, volume 183 of {\em Applied Mathematical Sciences}.
\newblock Springer, New York, 2013.

\bibitem{Cheskidov2016}
A.~Cheskidov, M.~C.~L. Filho, H.~J.~N. Lopes, and R.~Shvydkoy.
\newblock Energy conservation in two-dimensional incompressible ideal fluids.
\newblock {\em Comm. Math. Phys.}, 348(1):129--143, 2016.

\bibitem{Ciampa2022}
G.~Ciampa.
\newblock Energy conservation for 2{D} {E}uler with vorticity in {$L(\log L)^\alpha$}.
\newblock {\em Commun. Math. Sci.}, 20(3):855--875, 2022.

\bibitem{Constantin1994}
P.~Constantin, W.~E, and E.~S. Titi.
\newblock Onsager's conjecture on the energy conservation for solutions of {E}uler's equation.
\newblock {\em Comm. Math. Phys.}, 165(1):207--209, 1994.

\bibitem{DaPrato2014}
G.~Da~Prato and J.~Zabczyk.
\newblock {\em Stochastic equations in infinite dimensions}, volume 152 of {\em Encyclopedia of Mathematics and its Applications}.
\newblock Cambridge University Press, Cambridge, second edition, 2014.

\bibitem{DeLellis2009}
C.~De~Lellis and L.~Sz\'ekelyhidi, Jr.
\newblock The {E}uler equations as a differential inclusion.
\newblock {\em Ann. of Math. (2)}, 170(3):1417--1436, 2009.

\bibitem{DeRosa2024}
L.~{De Rosa} and J.~{Park}.
\newblock {No anomalous dissipation in two-dimensional incompressible fluids}.
\newblock {\em arXiv e-prints}, page arXiv:2403.04668, Mar. 2024.

\bibitem{DiNezza2012}
E.~Di~Nezza, G.~Palatucci, and E.~Valdinoci.
\newblock Hitchhiker's guide to the fractional {S}obolev spaces.
\newblock {\em Bull. Sci. Math.}, 136(5):521--573, 2012.

\bibitem{Eyink1994}
G.~L. {Eyink}.
\newblock {Energy dissipation without viscosity in ideal hydrodynamics I. Fourier analysis and local energy transfer}.
\newblock {\em Physica D Nonlinear Phenomena}, 78(3-4):222--240, Nov. 1994.

\bibitem{Flandoli1995}
F.~Flandoli and D.~Gatarek.
\newblock Martingale and stationary solutions for stochastic {N}avier-{S}tokes equations.
\newblock {\em Probab. Theory Related Fields}, 102(3):367--391, 1995.

\bibitem{Giri2024}
V.~{Giri} and R.-O. {Radu}.
\newblock {The Onsager conjecture in 2D: a Newton-Nash iteration}.
\newblock {\em Inventiones Mathematicae}, 238(2):691--768, Nov. 2024.

\bibitem{Gyongy1996}
I.~Gy\"ongy and N.~Krylov.
\newblock Existence of strong solutions for {I}t\^o's stochastic equations via approximations.
\newblock {\em Probab. Theory Related Fields}, 105(2):143--158, 1996.

\bibitem{Holden2015}
H.~Holden and N.~H. Risebro.
\newblock {\em Front tracking for hyperbolic conservation laws}, volume 152.
\newblock Springer, 2015.

\bibitem{Isett2018}
P.~Isett.
\newblock A proof of {O}nsager's conjecture.
\newblock {\em Ann. of Math. (2)}, 188(3):871--963, 2018.

\bibitem{Jin2024}
F.~{Jin}, S.~{Lanthaler}, M.~C. {Lopes Filho}, and H.~J. {Nussenzveig Lopes}.
\newblock {Sharp conditions for energy balance in two-dimensional incompressible ideal flow with external force}.
\newblock {\em arXiv e-prints}, page arXiv:2404.12572, Apr. 2024.

\bibitem{Kim2002}
J.~U. Kim.
\newblock On the stochastic {E}uler equations in a two-dimensional domain.
\newblock {\em SIAM J. Math. Anal.}, 33(5):1211--1227, 2002.

\bibitem{K41b}
A.~N. Kolmogorov.
\newblock Dissipation of energy in the locally isotropic turbulence.
\newblock {\em C. R. (Doklady) Acad. Sci. URSS (N.S.)}, 32:16--18, 1941.

\bibitem{K41a}
A.~N. Kolmogorov.
\newblock The local structure of turbulence in incompressible viscous fluid for very large {R}eynold's numbers.
\newblock {\em C. R. (Doklady) Acad. Sci. URSS (N.S.)}, 30:301--305, 1941.

\bibitem{K41c}
A.~N. Kolmogorov.
\newblock On degeneration of isotropic turbulence in an incompressible viscous liquid.
\newblock {\em C. R. (Doklady) Acad. Sci. URSS (N. S.)}, 31:538--540, 1941.

\bibitem{Kuksin2012}
S.~Kuksin and A.~Shirikyan.
\newblock {\em Mathematics of two-dimensional turbulence}, volume 194 of {\em Cambridge Tracts in Mathematics}.
\newblock Cambridge University Press, Cambridge, 2012.

\bibitem{Lanthaler2021}
S.~Lanthaler, S.~Mishra, and C.~Par\'{e}s-Pulido.
\newblock On the conservation of energy in two-dimensional incompressible flows.
\newblock {\em Nonlinearity}, 34(2):1084--1135, 2021.

\bibitem{LMP2021}
S.~Lanthaler, S.~Mishra, and C.~Parés-Pulido.
\newblock Statistical solutions of the incompressible euler equations.
\newblock {\em Mathematical Models and Methods in Applied Sciences}, 31(02):223–292, 02 2021.

\bibitem{Lopes2022}
M.~C. Lopes~Filho and H.~J. Nussenzveig~Lopes.
\newblock Energy balance for forced two-dimensional incompressible ideal fluid flow.
\newblock {\em Philos. Trans. Roy. Soc. A}, 380(2219):Paper No. 20210095, 12, 2022.

\bibitem{Onsager1949}
L.~Onsager.
\newblock Statistical hydrodynamics.
\newblock {\em Il Nuovo Cimento (1943-1954)}, 6(Suppl 2):279--287, 1949.

\bibitem{rohner2024}
T.~Rohner and S.~Mishra.
\newblock Efficient computation of large-scale statistical solutions to incompressible fluid flows.
\newblock In {\em Proceedings of the Platform for Advanced Scientific Computing Conference}, PASC ’24. ACM, June 2024.

\bibitem{Scheffer1993}
V.~Scheffer.
\newblock An inviscid flow with compact support in space-time.
\newblock {\em J. Geom. Anal.}, 3(4):343--401, 1993.

\bibitem{Shnirelman1997}
A.~Shnirelman.
\newblock On the nonuniqueness of weak solution of the {E}uler equation.
\newblock {\em Comm. Pure Appl. Math.}, 50(12):1261--1286, 1997.

\end{thebibliography}
\end{document}